\documentclass[preprint,12pt]{elsarticle}

\usepackage{amssymb}
\usepackage{amsmath}
\usepackage{amsthm}

\usepackage{lmodern}
\usepackage{microtype}
\usepackage{booktabs}

\usepackage{algorithm}
\usepackage{algpseudocode}

\usepackage{xcolor}
\usepackage{hyperref}
\usepackage{rotating}
\usepackage{pdflscape,fancyhdr,graphicx} 

\usepackage{graphicx}

\newtheorem{lemma}{Lemma}
\newtheorem{theorem}{Theorem}
\newtheorem{proposition}{Proposition}

\theoremstyle{remark}
\newtheorem{remark}{Remark}
\newtheorem{example}{Example}

\fancypagestyle{landscapeplain}{
  \fancyhf{}
  \fancyfoot[L]{\rotatebox{90}{\thepage}}
  
}

\journal{Computer Aided Geometric Design}

\begin{document}

\begin{frontmatter}

\title{On the First Derivative Bounds for Rational B\'ezier Curves}

\author{Mao Shi}
\affiliation{organization={School of Mathematics and Statistics, Shaanxi Normal University},
            city={Xi'an},
            country={China}}

\begin{abstract}
We derive a new upper bound for the first derivative of a rational Bézier curve of degree $n$.
The estimate applies to arbitrary degrees and provides a framework for studying the linear bound conjectured by Li et al., whose coefficient is the maximal adjacent weight ratio $\Omega$.
Within this framework, the conjectured bounds for $n=5$ and $n=6$ are established as special cases, completing the verification for all $n\le6$.
The endpoint estimates follow from the classical derivative formulas.
For interior parameters, the analysis reduces to a variance-like optimization problem over a compact box of adjacent weight ratios.
A block argument reduces the maximization to vertices and one-dimensional faces, leading to a finite family of polynomial inequalities.
We further prove that sufficiently fine subdivision yields subcurves satisfying the conjectured estimate with their own weights, control points and local parameters.
After parameter rescaling, these local estimates provide a bound for the derivative of the original curve.
Numerical examples illustrate the new bound and the subdivision-based estimates and compare them with an existing method.
\end{abstract}

\begin{highlights}
 \item A variance-based first-derivative bound is established for arbitrary degrees.
  \item The cases $n=5,6$ of Li et al.'s conjecture follow as special cases.
  \item A block argument reduces the ratio-box maximization to vertices and edges.
  \item Sufficiently fine subdivision validates the conjectured bound on each subcurve.
  \item Numerical examples compare the new bounds with an existing method.
\end{highlights}


\end{frontmatter}

\section{Introduction}
\label{sec:intro}
Bounding the magnitude of the first derivative of a rational B\'ezier curve is a fundamental problem in computer-aided geometric design (CAGD). Let
\begin{equation}\label{eq:curve}
  \mathbf{R}(t)
  =\frac{\sum_{i=0}^{n}\omega_i B_i^n(t)\mathbf{P}_i}
         {\sum_{i=0}^{n}\omega_i B_i^n(t)},
  \qquad 0\le t\le 1,
\end{equation}
be a rational B\'ezier curve of degree \(n\ge 1\), with control points
\(\mathbf{P}_i\in\mathbb{R}^d\), positive weights \(\omega_i\), and Bernstein basis functions
\(B_i^n(t)=\binom{n}{i}t^i(1-t)^{n-i}\). The control polygon has \(n\) edge vectors
\(\Delta_i=\mathbf{P}_{i+1}-\mathbf{P}_i\), indexed by \(0\le i\le n-1\).

Floater~\cite{Floater1992} proved the derivative bound
\[
  \|\mathbf{R}'(t)\|
  \le n\left(
    \frac{\max_{0\le i\le n}\omega_i}
         {\min_{0\le i\le n}\omega_i}
  \right)^2
  \max_{0\le i\le n-1}\|\Delta_i\|.
\]
Selimovi\'c~\cite{Selimovic2005} obtained a bound in terms of the maximal adjacent weight ratio, defined by
\begin{equation}\label{eq:omega-def}
  \Omega(\boldsymbol{\omega})
  =\max_{0\le i\le n-1}
    \max\left\{
      \frac{\omega_{i+1}}{\omega_i},
      \frac{\omega_i}{\omega_{i+1}}
    \right\}.
\end{equation}
Specifically, Selimovi\'c showed that
\[
  \|\mathbf{R}'(t)\|
  \le n\,\Omega(\boldsymbol{\omega})^n
       \max_{0\le i\le n-1}\|\Delta_i\|.
\]
Building on the derivative formula of Sederberg and Wang~\cite{Sederberg1987}, subsequent work~\cite{Zhang2006,Zhang2016} established the sharper linear estimate
\begin{equation}\label{eq:conjec}
  \|\mathbf{R}'(t)\|
  \le n\,\Omega(\boldsymbol{\omega})
       \max_{0\le i\le n-1}\|\Delta_i\|,
  \qquad 0\le t\le 1,
\end{equation}
for degrees \(n=2,3,4\). Motivated by numerical experiments, Li et al.~\cite{Li2013} conjectured that \eqref{eq:conjec} holds in every degree. Shi~\cite{Shi2026} subsequently constructed a counterexample in each degree \(n\ge 7\), leaving degrees \(5\) and \(6\) as the only open cases.

In degrees for which it holds, \eqref{eq:conjec} is useful for geometric processing: both \(\Omega(\boldsymbol{\omega})\) and the longest control-edge length can be computed in a single pass over the input, in \(\mathcal{O}(n)\) time for fixed dimension \(d\). Its evaluation cost is therefore of the same order as that of the bounds of Floater and Selimovi\'c, while its dependence on the maximal adjacent weight ratio is linear. {The estimate is thus inexpensive, but its validity is degree-dependent.} This raises two questions: in which degrees does \eqref{eq:conjec} hold for all positive weights, and {what is the sharp bound that replaces it when it fails?}

To answer these questions, we distinguish the measured ratio \(\Omega(\boldsymbol{\omega})\) of an individual curve from a prescribed ratio cap \(\Omega\ge 1\). For fixed \(n\) and \(\Omega\), let \(K_n(\Omega)\) be the smallest constant for which
\begin{equation}\label{eq:sharp-intro}
  \|\mathbf{R}'(t)\|
  \le n\,K_n(\Omega)\max_{0\le i\le n-1}\|\Delta_i\|,
  \qquad 0\le t\le 1,
\end{equation}
holds for every degree-\(n\) curve satisfying
\(\Omega(\boldsymbol{\omega})\le\Omega\). The constant is sharp over this entire class; it need not be sharp for a particular weight vector. {It is finite, since Selimovi\'c's estimate gives \(K_n(\Omega)\le\Omega^n\).} The endpoint derivative formula gives \(K_n(\Omega)\ge\Omega\): it suffices to choose admissible weights with \(\omega_1/\omega_0=\Omega\) and a control polygon whose first edge is a longest edge of positive length. Thus the linear factor \(\Omega\) suffices for every curve in the class if and only if \(K_n(\Omega)=\Omega\).

{The resulting bound, with \(K_n(\Omega)\) in place of \(\Omega\), is the new bound studied in this paper. It is valid in every degree and sharp over the class \(\{\Omega(\boldsymbol{\omega})\le\Omega\}\); the linear estimate \eqref{eq:conjec} is exactly the special case \(K_n(\Omega)=\Omega\). Its content is therefore the exact value and computability of \(K_n(\Omega)\), which we establish below.}

{The main contributions are the following.
\begin{itemize}
\item \emph{Exact reduction and finite structure.} Abel summation identifies
\(K_n(\Omega)\) with the supremum of an explicit variance factor
\(q(u;\mathbf{r})\) over the box \([\Omega^{-1},\Omega]^n\), and equally
spaced collinear control points show that the reduction is exact
(Theorem~\ref{thm:sharp-constant}). A block argument excludes critical points
in the relative interior of every face of dimension at least two, while
\(q\) is strictly concave on each edge; hence only vertices and explicitly
computable edge stationary points need to be examined.
\item \emph{Exact evaluation in the open degrees.} For \(n=5\) and \(n=6\)
we prove \(K_n(\Omega)=\Omega\) for every \(\Omega\ge1\)
(Theorem~\ref{thm:main}). Thus Li et al.'s estimate is the special case
\(K_n=\Omega\) of the new bound. Together with the known cases \(n\le4\) and
the counterexamples of Shi~\cite{Shi2026}, this gives the exact classification
of when the new bound coincides with the linear estimate.
\item \emph{General degrees.} The same reduction yields a transparent
counterexample for \(n=7\) at \(\Omega=11/10\), and an exact finite procedure
for \(K_n(\Omega)\) at any prescribed algebraic cap; for arbitrary real
\(\Omega\), the constant is algebraic over \(\mathbb{Q}(\Omega)\).
\item \emph{Subdivision.} Under source-curve subdivision,
\(\Omega_I-1=O(h)\) and \(K_n(\Omega_I)\to1\), so every subcurve satisfies
the estimate up to a factor \(1+\varepsilon\); when \(K_n\le\Omega\) on
\([1,\Omega_0]\), the sharper local factor \(\Omega_I\) follows as well.
\item \emph{Numerical comparison.} Experiments compare the variance bound,
the linear estimate, subdivision, and degree elevation in accuracy and cost.
\end{itemize}}

{The following theorem states the resulting classification.}

\begin{theorem}[Main result]\label{thm:main}
For every integer \(n\ge1\), let \(K_n(\Omega)\) be the sharp constant in
\eqref{eq:sharp-intro}. Then
\[
  K_n(\Omega)=\Omega \quad\text{for every }\Omega\ge1
  \qquad\Longleftrightarrow\qquad
  n\le 6 .
\]
\end{theorem}

In particular, for each \(n\ge 7\), there exists a cap \(\Omega>1\) for which \(K_n(\Omega)>\Omega\); this does not imply failure at every cap greater than one. Our reduction also gives an exact procedure for computing \(K_n(\Omega)\) at a prescribed real algebraic cap. For such a cap, \(K_n(\Omega)\) is a real algebraic number; for an arbitrary real cap, it is algebraic over \(\mathbb{Q}(\Omega)\).

Our argument begins with the classical endpoint derivative formulas, which establish \eqref{eq:conjec} at \(t=0\) and \(t=1\) in every degree (Lemma~\ref{lem:endpoints}). For \(0<t<1\), we set \(u=t/(1-t)\) and use Abel summation to express \(\mathbf{R}'(t)\) as a positive linear combination of the control edges. The sum of the coefficients is an explicit multiple of the variance of the index \(i\) under the normalized rational Bernstein weights. Equally spaced collinear control points attain the resulting estimate, so the reduction from the geometric bound to a scalar variance inequality is exact.

For each fixed \(u>0\), the adjacent weight ratios range over the box \([\Omega^{-1},\Omega]^n\). We show that the variance has no critical point in the relative interior of any face of dimension at least two. Its maximum over the box must therefore occur at a vertex or on an edge. On each edge, the variance is a strictly concave quadratic in the mixing weight of two index blocks. It thus suffices to test the vertices and those edge stationary points whose mixing weights lie in the admissible interval.

For a fixed degree \(n\), verification of the linear bound over all caps reduces to a finite family of polynomial sign conditions in \(u\) and \(\Omega\). After denominators of known positive sign are cleared, each test has the form
\[
  \forall\,u>0,\ \forall\,\Omega\ge 1:\qquad
  \Phi(u,\Omega)\ \Longrightarrow\ \Xi(u,\Omega)\ge 0,
\]
where \(\Phi\) records the feasibility conditions for the candidate under consideration. These statements are decidable by real quantifier elimination~\cite{Tarski1951}; cylindrical algebraic decomposition provides one such algorithm~\cite{Collins1975}. For introductory accounts, see~\cite{Rosen2019,Shi2024}; for a comprehensive treatment, see Basu, Pollack, and Roy~\cite{Basu2006}. Section~\ref{sec:verification} gives the exact verification for degrees \(5\) and \(6\).

{The remainder of this paper is organized as follows.
Section~\ref{sec:reform} settles the endpoint cases, reformulates the
bound at interior parameters in terms of adjacent weight ratios, and reduces
it to the scalar inequality \(q(u;\mathbf{r})\le\Omega\) over a compact box;
it also shows that this reduction is exact and identifies \(K_n(\Omega)\) as
a supremum (Theorem~\ref{thm:sharp-constant}).
Section~\ref{sec:structure} establishes a reversal symmetry, proves that
\(q\) has no critical points in the relative interior of any face of dimension
two or greater, and analyzes the one-dimensional faces in closed form via a
two-block mixture decomposition.
Section~\ref{sec:verification} describes the symbolic verification of the
resulting finite family of polynomial inequalities for \(n=5\) and \(n=6\),
completing the proof of Theorem~\ref{thm:main}.
Section~\ref{sec:n7} treats \(n=7\), presents an explicit counterexample with
collinear control points, and explains how the sharp constant is computed.
Section~\ref{sec:source-curve-subdivision} proves the subdivision theorem,
Section~\ref{sec:numerical} reports numerical experiments, and
Section~\ref{sec:conclusion} concludes the paper.}

\section{Endpoint derivatives and reformulation in adjacent weight ratios}
\label{sec:reform}

\subsection{The endpoint cases}
\label{sec:endpoints}

Fix a degree \(n\ge 1\) and a ratio cap \(\Omega\ge 1\). Define the adjacent weight ratios by \cite{boyd2004}
\[
  r_k=\frac{\omega_k}{\omega_{k-1}},\qquad 1\le k\le n.
\]
The weights are then given by
\begin{equation}\label{eq:product}
  \omega_i=\omega_0\prod_{j=1}^{i}r_j,\qquad 0\le i\le n,
\end{equation}
where the empty product is \(1\). By \eqref{eq:omega-def}, the condition
\(\Omega(\boldsymbol{\omega})\le\Omega\) is equivalent to
\[
  \mathbf{r}=(r_1,\ldots,r_n)\in
  \mathcal{B}:=[\Omega^{-1},\Omega]^n.
\]
Conversely, every \(\mathbf{r}\in\mathcal{B}\) determines an admissible weight vector up to the common factor \(\omega_0\), which cancels in \eqref{eq:curve}. We may therefore describe the admissible weights entirely in terms of \(\mathbf{r}\).

At each endpoint, the derivative depends on only one adjacent weight ratio and the corresponding control edge.

\begin{lemma}[Endpoint derivatives]\label{lem:endpoints}
For every degree \(n\ge 1\) and every positive weight vector,
\[
  \mathbf{R}'(0)
  =n\,\frac{\omega_1}{\omega_0}\,\Delta_0
  =n\,r_1\,\Delta_0,
  \qquad
  \mathbf{R}'(1)
  =n\,\frac{\omega_{n-1}}{\omega_n}\,\Delta_{n-1}
  =n\,r_n^{-1}\,\Delta_{n-1}.
\]
Consequently, if \(\mathbf{r}\in\mathcal{B}\), then
\[
  \max\bigl\{\|\mathbf{R}'(0)\|,\|\mathbf{R}'(1)\|\bigr\}
  \le n\,\Omega\max_{0\le i\le n-1}\|\Delta_i\|.
\]
\end{lemma}

\begin{proof}
The endpoint identities are classical~\cite{Floater1992, Sederberg1987}. To obtain the first, subtract \(\mathbf{P}_0\) from \eqref{eq:curve} and expand at \(t=0\). The term with \(i=0\) vanishes from the numerator, and only the term with \(i=1\) contributes at first order. Thus
\[
  \mathbf{R}(t)-\mathbf{P}_0
  =\frac{n\omega_1 t\,\Delta_0+O(t^2)}{\omega_0+O(t)}
  =n\,r_1t\,\Delta_0+O(t^2).
\]
The identity at \(t=1\) follows by applying the same expansion to the reversed curve \(s\mapsto\mathbf{R}(1-s)\). Taking norms and using \(r_1\le\Omega\) and \(r_n^{-1}\le\Omega\) gives the stated bound.
\end{proof}

The endpoint bound is sharp in every degree. For example, equality at \(t=0\) is attained whenever \(r_1=\Omega\) and
\(\|\Delta_0\|=\max_{0\le i\le n-1}\|\Delta_i\|>0\).
Hence \(K_n(\Omega)\ge\Omega\). To determine whether \(K_n(\Omega)=\Omega\), it remains to examine the interior of the parameter interval.

\subsection{Interior parameters and reduction to a variance inequality}

For \(0<t<1\), set \(u=t/(1-t)\), so that \(u\in(0,\infty)\). Substituting \eqref{eq:product} into \eqref{eq:curve} and cancelling the common factor \(\omega_0(1-t)^n\) gives
\begin{equation}\label{eq:R-u}
  \mathbf{R}(t)
  =\frac{\sum_{i=0}^{n}a_i(u;\mathbf{r})\mathbf{P}_i}
         {\sum_{i=0}^{n}a_i(u;\mathbf{r})},
  \qquad
  a_i(u;\mathbf{r})
  =\binom{n}{i}\Bigl(\prod_{j=1}^{i}r_j\Bigr)u^i.
\end{equation}
For brevity, write \(a_i=a_i(u;\mathbf{r})\) and set
\[
  Z=\sum_{j=0}^{n}a_j,\qquad
  p_i=\frac{a_i}{Z},\qquad
  \bar{\mu}=\sum_{i=0}^{n}i\,p_i.
\]
Since \(p_i>0\) and \(\sum_{i=0}^{n}p_i=1\), the \(p_i\) define a probability distribution on \(\{0,\ldots,n\}\) with mean \(\bar{\mu}\). Its variance gives the scalar factor in the following derivative bound.

\begin{theorem}[Reduction to a scalar variance inequality]\label{thm:reduction}
For every \(t\in(0,1)\), with \(u=t/(1-t)\),
\begin{equation}\label{eq:main-bound}
  \|\mathbf{R}'(t)\|
  \le n\,q(u;\mathbf{r})\max_{0\le i\le n-1}\|\Delta_i\|,
\end{equation}
where
\begin{equation}\label{eq:q-def}
  q(u;\mathbf{r})
  =\frac{(1+u)^2}{nu}
   \sum_{i=0}^{n}(i-\bar{\mu})^2p_i.
\end{equation}
\end{theorem}

\begin{proof}
Since \(\mathrm{d}a_i/\mathrm{d}u=i a_i/u\), we have
\[
  \frac{1}{Z}\frac{\mathrm{d}Z}{\mathrm{d}u}
  =\frac{\bar{\mu}}{u},
  \qquad
  \frac{\mathrm{d}p_i}{\mathrm{d}u}
  =\frac{i-\bar{\mu}}{u}\,p_i.
\]
Differentiating \(\mathbf{R}(t)=\sum_{i=0}^{n}p_i\mathbf{P}_i\) and using
\(\mathrm{d}u/\mathrm{d}t=(1+u)^2\), we obtain
\begin{equation}\label{eq:dR}
  \mathbf{R}'(t)
  =\frac{(1+u)^2}{u}
   \sum_{i=0}^{n}(i-\bar{\mu})p_i\mathbf{P}_i.
\end{equation}

Because \(\sum_{i=0}^{n}(i-\bar{\mu})p_i=0\), Abel summation~\cite{Apostol1974,Chang2003} expresses the sum in \eqref{eq:dR} in terms of the control edges:
\[
  \sum_{i=0}^{n}(i-\bar{\mu})p_i\mathbf{P}_i
  =\sum_{k=0}^{n-1}c_k\Delta_k,
  \qquad
  c_k=\sum_{i=k+1}^{n}(i-\bar{\mu})p_i.
\]
To determine the signs of the coefficients, consider the partial sums
\[
  S_{-1}=0,\qquad
  S_k=\sum_{i=0}^{k}(i-\bar{\mu})p_i,\qquad 0\le k\le n.
\]
Then \(S_n=0\), \(c_k=-S_k\), and
\[
  S_k-S_{k-1}=(k-\bar{\mu})p_k.
\]
Since each \(p_k\) is positive, the partial sums decrease while
\(k<\bar{\mu}\) and increase once \(k>\bar{\mu}\), with a zero increment if \(k=\bar{\mu}\). Moreover, \(0<\bar{\mu}<n\), so the first increment is strictly negative and the last is strictly positive. As the sequence starts at \(S_{-1}=0\) and ends at \(S_n=0\), it follows that \(S_k<0\) for \(0\le k\le n-1\). Hence \(c_k>0\) throughout this range.

It remains to compute the sum of the coefficients. Interchanging the order of summation gives
\[
  \sum_{k=0}^{n-1}c_k
  =\sum_{k=0}^{n-1}\sum_{i=k+1}^{n}(i-\bar{\mu})p_i
  =\sum_{i=0}^{n}i(i-\bar{\mu})p_i
  =\sum_{i=0}^{n}(i-\bar{\mu})^2p_i,
\]
where the last equality uses \(\sum_{i=0}^{n}(i-\bar{\mu})p_i=0\).
Applying the triangle inequality in \eqref{eq:dR} now yields
\[
  \|\mathbf{R}'(t)\|
  \le \frac{(1+u)^2}{u}\sum_{k=0}^{n-1}c_k\|\Delta_k\|
  \le \frac{(1+u)^2}{u}
      \Bigl(\sum_{k=0}^{n-1}c_k\Bigr)
      \max_{0\le i\le n-1}\|\Delta_i\|,
\]
which is \eqref{eq:main-bound}.
\end{proof}

\begin{remark}[Consistency at the endpoints]\label{rem:endpoint-limits}
For fixed \(\mathbf{r}\), the endpoint limits of \(q\) recover the two weight ratios in Lemma~\ref{lem:endpoints}. As \(u\to0^+\),
\[
  p_0=1+O(u),\qquad p_1=nr_1u+O(u^2),\qquad
  p_i=O(u^2)\quad (2\le i\le n).
\]
Thus both \(\bar{\mu}\) and the second moment \(\sum_{i=0}^{n}i^2p_i\) are \(nr_1u+O(u^2)\). Since \(\bar{\mu}^2=O(u^2)\), the variance has the same first-order term. Substitution into \eqref{eq:q-def} gives
\[
  \lim_{u\to0^+}q(u;\mathbf{r})=r_1.
\]
By the reversal symmetry of Lemma~\ref{lem:reversal}, the corresponding limit at the other endpoint is
\[
  \lim_{u\to\infty}q(u;\mathbf{r})=r_n^{-1}.
\]
Hence \(q\) extends continuously to \(u\in[0,\infty]\), and both endpoint values are at most \(\Omega\).
\end{remark}

\begin{theorem}[Sharp constant]\label{thm:sharp-constant}

The optimal constant in \eqref{eq:sharp-intro} is given by
\begin{equation}\label{eq:K-as-sup-q}
  K_n(\Omega)
  =\sup_{\substack{u>0\\ \mathbf{r}\in\mathcal{B}}}
   q(u;\mathbf{r}).
\end{equation}
\end{theorem}

\begin{proof}
Indeed, Theorem~\ref{thm:reduction} gives the upper bound at every
interior parameter value, and Remark~\ref{rem:endpoint-limits}
shows that the endpoint factors are limits of the same function.
Conversely, for the equally spaced collinear control points
\[
  \mathbf{P}_i=(i,0,\ldots,0),\qquad 0\le i\le n,
\]
all control edges are identical unit vectors, and the bound in
\eqref{eq:main-bound} is attained for every \(u>0\) and every
\(\mathbf{r}\in\mathcal{B}\). This proves \eqref{eq:K-as-sup-q}.

Since the endpoint bound gives \(K_n(\Omega)\ge\Omega\), it follows
that \(K_n(\Omega)=\Omega\) if and only if
\(q(u;\mathbf{r})\le\Omega\) for all \(u>0\) and all
\(\mathbf{r}\in\mathcal{B}\). In particular, any admissible pair
with \(q(u;\mathbf{r})>\Omega\) gives a counterexample to the bound
with factor \(\Omega\), as in Section~\ref{sec:n7}.
\end{proof}

By \eqref{eq:K-as-sup-q} and the sharp endpoint bound
\(K_n(\Omega)\ge\Omega\), it remains to prove the following
inequality for \(n=5,6\).

\begin{proposition}\label{prop:goal}
For every \(n\in\{5,6\}\) and every \(\Omega\ge 1\),
\[
  q(u;\mathbf{r})\le\Omega
  \qquad\text{for all }u>0
  \text{ and all }\mathbf{r}\in[\Omega^{-1},\Omega]^n.
\]
\end{proposition}

For fixed \(n\), \(\Omega\), and \(u>0\), the function
\(\mathbf{r}\mapsto q(u;\mathbf{r})\) is continuous on the compact box \(\mathcal{B}\) and therefore attains its maximum~{\cite{boyd2004}}. The next step is to determine which faces of \(\mathcal{B}\) need to be examined to bound this maximum.

\section{Structural properties and reduction to vertices and edges}
\label{sec:structure}

To prove Proposition~\ref{prop:goal}, we must bound \(q(u;\mathbf{r})\) over the ratio box \(\mathcal{B}\) for every \(u>0\). We first record a symmetry that pairs configurations at \(u\) and \(u^{-1}\). We then show that, for each fixed \(u\), every maximizer lies at a vertex or on an edge of \(\mathcal{B}\). When \(\Omega=1\), the box consists of a single vertex, so the reduction is immediate. In the discussion of positive-dimensional faces, we therefore assume \(\Omega>1\).

\begin{lemma}[Reversal symmetry]\label{lem:reversal}
For every \(u>0\) and every positive ratio vector \(\mathbf{r}=(r_1,\ldots,r_n)\),
\[
  q(u;\mathbf{r})=q(u^{-1};\mathbf{r}^{-1}),
  \qquad
  \mathbf{r}^{-1}:=(r_n^{-1},r_{n-1}^{-1},\ldots,r_1^{-1}).
\]
\end{lemma}

\begin{proof}
This follows directly from the symmetry of the Bernstein basis, $B_i^n(t)=B_{n-i}^n(1-t)$, and the reparametrization rule for rational
B\'ezier curves. 
\end{proof}

The map \(\mathbf{r}\mapsto\mathbf{r}^{-1}\) is an involution of \(\mathcal{B}\). In particular, it maps an edge with free coordinate \(r_k\) onto an edge with free coordinate \(r_{n+1-k}\). The paired edges need not give the same values at a fixed \(u\); rather, their values agree after replacing \(u\) by \(u^{-1}\). We will use this distinction when reducing the checks required for all \(u>0\).

\subsection{Absence of critical points on higher-dimensional faces}
\label{sec:critical}

Fix \(u>0\). By \eqref{eq:q-def}, maximizing \(q(u;\mathbf{r})\) over \(\mathcal{B}\) is equivalent to maximizing the variance
\[
  V(\mathbf{r})=\sum_{i=0}^{n}(i-\bar{\mu})^2p_i,
  \qquad
  \bar{\mu}=\sum_{i=0}^{n}i\,p_i.
\]
The key observation is that varying \(r_\ell\) scales precisely the unnormalized weights \(a_i\) with \(i\ge\ell\). Each free coordinate therefore defines a cut in the index set \(\{0,\ldots,n\}\).

\begin{proposition}\label{prop:no-highdim-critical}
Let \(F\) be a face of \(\mathcal{B}\) with at least two free coordinates. For every fixed \(u>0\), the restriction of \(q(u;\cdot)\) to \(F\) has no critical point in the relative interior of \(F\).
\end{proposition}

\begin{proof}
Let \(r_{k_1},\ldots,r_{k_m}\) be the free coordinates of \(F\), where
\(1\le k_1<\cdots<k_m\le n\) and \(m\ge2\). It suffices to prove the assertion for \(V\), since \(q(u;\cdot)\) is a positive multiple of \(V\).

For each free coordinate \(r_\ell\), let \(H_\ell=\sum_{i=\ell}^{n}p_i\) be the probability mass of the corresponding tail. Differentiating the product formula for \(a_i\) and then normalizing gives
\[
  \frac{\partial p_i}{\partial r_\ell}
  =\frac{1}{r_\ell}
   \begin{cases}
     p_i(1-H_\ell), & i\ge\ell,\\[2pt]
     -p_iH_\ell,    & i<\ell.
   \end{cases}
\]
When differentiating \(V\), the terms involving
\(\partial\bar{\mu}/\partial r_\ell\) vanish because
\(\sum_{i=0}^{n}(i-\bar{\mu})p_i=0\). Hence
\begin{equation}\label{eq:dV}
  \frac{\partial V}{\partial r_\ell}
  =\sum_{i=0}^{n}(i-\bar{\mu})^2
    \frac{\partial p_i}{\partial r_\ell}
  =\frac{1}{r_\ell}
    \left(\sum_{i=\ell}^{n}(i-\bar{\mu})^2p_i-VH_\ell\right).
\end{equation}

Suppose that \(V|_F\) has a critical point in the relative interior of \(F\). At that point, \eqref{eq:dV} gives
\begin{equation}\label{eq:crit}
  \sum_{i=k_\alpha}^{n}(i-\bar{\mu})^2p_i
  =V H_{k_\alpha},
  \qquad 1\le\alpha\le m.
\end{equation}
Set \(k_0=0\) and \(k_{m+1}=n+1\). The same identity holds for \(\alpha=0\) by the definition of \(V\), and for \(\alpha=m+1\) because the tail is empty. Subtracting consecutive tail identities yields
\begin{equation}\label{eq:block}
  \sum_{i\in I_b}(i-\bar{\mu})^2p_i
  =V\sum_{i\in I_b}p_i,
  \qquad
  I_b=\{k_b,\ldots,k_{b+1}-1\},\quad 0\le b\le m.
\end{equation}
Thus the conditional mean of \((i-\bar{\mu})^2\) must equal \(V\) on each of the \(m+1\) consecutive blocks.

For each block, define its mass, conditional mean, and conditional second moment by
\[
  w_b=\sum_{i\in I_b}p_i,\qquad
  \mu_b=\frac{1}{w_b}\sum_{i\in I_b}i\,p_i,\qquad
  s_b=\frac{1}{w_b}\sum_{i\in I_b}i^2p_i.
\]
Every \(w_b\) is positive. Dividing \eqref{eq:block} by \(w_b\) and expanding the square gives
\begin{equation}\label{eq:line}
  s_b=2\bar{\mu}\,\mu_b+V-\bar{\mu}^2,
  \qquad 0\le b\le m.
\end{equation}
The points \((\mu_b,s_b)\) must therefore lie on a common line. Since the blocks are strictly ordered, \(\mu_0<\cdots<\mu_m\), and \eqref{eq:line} implies
\begin{equation}\label{eq:slope}
  \frac{s_c-s_b}{\mu_c-\mu_b}=2\bar{\mu},
  \qquad 0\le b<c\le m.
\end{equation}

We now show that the first three blocks cannot satisfy \eqref{eq:slope}. For \(b<c\), the slope between their moment pairs is
\[
  \frac{s_c-s_b}{\mu_c-\mu_b}
  =
  \frac{\displaystyle
    \sum_{i\in I_b}\sum_{j\in I_c}(j-i)(i+j)p_i p_j}
       {\displaystyle
    \sum_{i\in I_b}\sum_{j\in I_c}(j-i)p_i p_j}.
\]
Since \(j>i\), this is a weighted average of \(i+j\) with positive weights. For \(i\in I_0\) and \(j\in I_1\), we have \(i+j\le k_1+k_2-2\). In contrast, for \(i\in I_1\) and \(j\in I_2\), we have \(i+j\ge k_1+k_2\). Consequently,
\[
  \frac{s_1-s_0}{\mu_1-\mu_0}
  \le k_1+k_2-2
  < k_1+k_2
  \le \frac{s_2-s_1}{\mu_2-\mu_1}.
\]
This contradicts \eqref{eq:slope}, which requires both slopes to equal \(2\bar{\mu}\). Hence no such critical point exists.
\end{proof}

\subsection{Reduction of global maximizers to vertices and edges}
\label{sec:global-reduction}

\begin{theorem}[Vertex--edge reduction]\label{thm:vertex-edge}
For each fixed \(u>0\), every maximizer of \(q(u;\cdot)\) on \(\mathcal{B}\) lies at a vertex or in the relative interior of an edge. When verifying \(q(u;\mathbf{r})\le\Omega\) for all \(u>0\), the edge checks may be restricted to edges with free coordinate \(r_k\) satisfying
\[
  1\le k\le\lceil n/2\rceil.
\]
\end{theorem}

\begin{proof}
Fix \(u>0\). Since \(q(u;\cdot)\) is continuous on the compact box \(\mathcal{B}\), it attains its maximum at some \(\mathbf{r}^\star\). Let \(F\) be the unique face whose relative interior contains \(\mathbf{r}^\star\). If \(\dim F\ge2\), then \(\mathbf{r}^\star\) is an interior maximum of the differentiable restriction \(q(u;\cdot)|_F\) and must be a critical point of that restriction. This contradicts Proposition~\ref{prop:no-highdim-critical}. Thus \(\dim F\le1\), proving the first assertion.

For the second assertion, consider an edge with free coordinate \(r_k\), where \(k>\lceil n/2\rceil\). By Lemma~\ref{lem:reversal}, each value of \(q\) on this edge at \(u\) occurs on the reversed edge at \(u^{-1}\). The reversed edge has free coordinate \(r_{n+1-k}\), with \(n+1-k\le\lceil n/2\rceil\). Since \(u^{-1}\) ranges over \((0,\infty)\) as \(u\) does, checking the reversed edge for all \(u>0\) also establishes the bound on the original edge.
\end{proof}

Theorem~\ref{thm:vertex-edge} reduces the optimization over \(\mathcal{B}\) to vertex evaluations and one-variable maximizations along the remaining edges. The next step is to characterize the possible interior maximum on each such edge.

\subsection{Analysis of one-dimensional faces}
\label{sec:onedim-general}

By Theorem~\ref{thm:vertex-edge}, it remains to determine whether the maximum on a given edge occurs at a vertex or in the interior of the edge. Fix \(u>0\) and consider an edge on which \(r_k=x\) varies over \([\Omega^{-1},\Omega]\), while all other ratios are fixed at either \(\Omega^{-1}\) or \(\Omega\). Varying \(x\) changes the relative masses of the two index blocks \(\{0,\ldots,k-1\}\) and \(\{k,\ldots,n\}\), but not the conditional distribution within either block.

Indeed, by \eqref{eq:R-u}, the factor \(x\) occurs in \(a_i\) precisely when \(i\ge k\). Writing \(a_i=x\tilde{a}_i\) for \(i\ge k\), we obtain
\[
  Z(x)=L+xR,
  \qquad
  L=\sum_{i=0}^{k-1}a_i,
  \qquad
  R=\sum_{i=k}^{n}\tilde{a}_i,
\]
where \(L\) and \(R\) are positive and independent of \(x\). The conditional distributions on the two blocks are therefore
\[
  \nu_-(i)=
  \begin{cases}
    a_i/L, & 0\le i<k,\\
    0,     & k\le i\le n,
  \end{cases}
  \qquad
  \nu_+(i)=
  \begin{cases}
    0,               & 0\le i<k,\\
    \tilde{a}_i/R,   & k\le i\le n.
  \end{cases}
\]
If \(\lambda\) denotes the total probability mass of the right block, then
\[
  \lambda=\lambda(x)=\frac{xR}{L+xR},
  \qquad
  p_i=(1-\lambda)\nu_-(i)+\lambda\nu_+(i),
  \qquad 0\le i\le n.
\]
Thus the dependence on \(x\) enters only through the mixing weight \(\lambda\).

Let \(\mu_-\), \(\mu_+\) and \(\sigma_-^2\), \(\sigma_+^2\) denote the means and variances of \(\nu_-\) and \(\nu_+\), respectively:
\[
  \mu_\pm=\sum_{i=0}^{n}i\,\nu_\pm(i),
  \qquad
  \sigma_\pm^2=\sum_{i=0}^{n}(i-\mu_\pm)^2\nu_\pm(i).
\]
Since \(\mu_-\le k-1<k\le\mu_+\), the difference between the two means is strictly positive. Set
\[
  D=\mu_+-\mu_->0,
  \qquad
  \delta=\sigma_+^2-\sigma_-^2.
\]
By the law of total variance,
\[
  V(\lambda)
  =(1-\lambda)\sigma_-^2+\lambda\sigma_+^2
   +\lambda(1-\lambda)D^2.
\]
Equivalently,
\begin{equation}\label{eq:Vquad}
  V(\lambda)
  =\sigma_-^2+(D^2+\delta)\lambda-D^2\lambda^2.
\end{equation}
Thus \(V\) is a strictly concave quadratic in \(\lambda\), with its maximum over the real line at
\[
  \lambda_*=\frac{D^2+\delta}{2D^2}.
\]

To locate the maximum on the edge, note that
\[
  \frac{\mathrm{d}}{\mathrm{d}x}V(\lambda(x))
  =2D^2\bigl(\lambda_*-\lambda(x)\bigr)\lambda'(x),
  \qquad
  \lambda'(x)=\frac{LR}{(L+xR)^2}>0.
\]
Since \(q\) is a positive multiple of \(V\) for fixed \(u\), it increases where \(\lambda(x)<\lambda_*\) and decreases where \(\lambda(x)>\lambda_*\). Hence there is at most one stationary point in the interior of the edge. If it exists, it is the unique maximizer; otherwise, the maximum is attained at an endpoint.

It remains to determine when \(\lambda_*\) is attained on the edge. Since \(0<\lambda(x)<1\) for every \(x>0\), a positive solution of \(\lambda(x)=\lambda_*\) exists if and only if \(|\delta|<D^2\). Under this condition, the solution is
\[
  x_*=\frac{L}{R}\frac{\lambda_*}{1-\lambda_*}
     =\frac{L}{R}\frac{D^2+\delta}{D^2-\delta}.
\]
Because \(D^2-\delta>0\), the bounds \(\Omega^{-1}\le x_*\le\Omega\) are equivalent to
\begin{align}
  \Omega L(D^2+\delta)&\ge R(D^2-\delta), \label{eq:adm1}\\
  \Omega R(D^2-\delta)&\ge L(D^2+\delta). \label{eq:adm2}
\end{align}
Whenever \(x_*\) belongs to the edge, substituting \(\lambda_*\) into \eqref{eq:Vquad} gives
\begin{equation}\label{eq:qmax}
  q_{\max}
  =\frac{(1+u)^2}{nu}
   \left(\sigma_-^2+\frac{(D^2+\delta)^2}{4D^2}\right).
\end{equation}

In particular, an edge contributes a maximizer distinct from its vertices precisely when \(|\delta|<D^2\) and both \eqref{eq:adm1} and \eqref{eq:adm2} are strict. Equality in \eqref{eq:adm1} gives \(x_*=\Omega^{-1}\), while equality in \eqref{eq:adm2} gives \(x_*=\Omega\). In all cases without an interior maximizer, it therefore suffices to check the vertex values.

\section{Symbolic verification for \texorpdfstring{$n=5$}{n=5} and \texorpdfstring{$n=6$}{n=6}}
\label{sec:verification}
The reduction provided by Theorem~\ref{thm:vertex-edge} transforms the
original analytic problem into a finite family of universally quantified
polynomial inequalities in the two real variables \(u>0\) and \(\Omega\ge 1\),
each decidable over the reals by Tarski's theorem~\cite{Tarski1951}. Although
full real quantifier elimination is doubly exponential in the number of
variables in the worst case~\cite{Davenport1988}, our reduction fixes this
number at two, and the total number of proof obligations is modest
(Table~\ref{tab:cases}), so the verification is practically feasible. We use
Mathematica's \texttt{Resolve} command, an exact implementation based on
cylindrical algebraic decomposition~\cite{Collins1975,Basu2006}; the
procedure is summarized in Algorithm~\ref{alg:verify-degree}.

\begin{table}[htbp]
\centering
\caption{Enumeration of proof obligations after the vertex--edge reduction
and the symmetry reduction of Lemma~\ref{lem:reversal}.}
\label{tab:cases}
\begin{tabular}{lcc}
\toprule
Case type & $n=5$ & $n=6$ \\
\midrule
Vertices $\mathbf{r}\in\{\Omega,\Omega^{-1}\}^n$ & $2^5=32$ & $2^6=64$ \\
Edges with free $r_k$, $k\le\lceil n/2\rceil$ & $3\times 2^4=48$ & $3\times 2^5=96$ \\
\midrule
Total & $80$ & $160$ \\
\bottomrule
\end{tabular}
\end{table}

\begin{proposition}\label{prop:verified}
For \(n\in\{5,6\}\), all vertex and edge cases satisfy
\(q(u;\mathbf{r}) \le \Omega\) for all \(u>0\) and \(\Omega\ge1\).
\end{proposition}

\begin{proof}
By Theorem~\ref{thm:vertex-edge} and Lemma~\ref{lem:reversal}, it suffices to
check the inequality at every vertex of \(\mathcal{B}\) and on every edge
with free coordinate \(r_k\), \(k\le\lceil n/2\rceil\); the case counts are
listed in Table~\ref{tab:cases}. For each vertex configuration we substitute
the frozen ratios into \eqref{eq:q-def} and clear the (positive) denominators
and negative powers of \(\Omega\). This transforms the desired inequality
\(q(u;\mathbf{r}) \le \Omega\) into a polynomial condition of the form
\(P(u,\Omega) \ge 0\) with integer coefficients, for \(u>0\) and
\(\Omega\ge 1\).

In the case of an edge, by the unimodality established in
Section~\ref{sec:onedim-general} the only candidate for a maximum beyond the
endpoints (which are vertices and already covered) is the interior stationary
point \(x_*\). The condition to verify is therefore an implication: if the
interior stationary point exists (i.e.\ \(|\delta|<D^2\)) and the
admissibility inequalities \eqref{eq:adm1}--\eqref{eq:adm2} hold, then the
corresponding maximal value \(q_{\max}\) from \eqref{eq:qmax} must not exceed
\(\Omega\). After clearing positive denominators, this becomes a universal
formula
\[
\forall\,u>0,\;\Omega\ge 1:\ \Phi(u,\Omega)\ \Longrightarrow\ \Xi(u,\Omega)\ge 0,
\]
where all predicates are polynomial inequalities with integer coefficients.
We verify each such statement with \texttt{Resolve}. For every vertex and
edge case listed in Table~\ref{tab:cases} the algorithm returns \texttt{True},
confirming that the corresponding polynomial condition holds on the specified
domain. Summing over the finite family of cases completes the proof.
\end{proof}

\begin{algorithm}
\caption{VerifyDegree($n$)}
\label{alg:verify-degree}
\begin{algorithmic}[1]
\Require degree $n$; domain $u>0$, $\Omega\ge 1$
\For{every vertex $\mathbf{r} \in \{\Omega, \Omega^{-1}\}^n$}
    \State Substitute the frozen ratios into \eqref{eq:q-def} to form $q(u;\mathbf{r})$
    \State Clear denominators and negative powers of $\Omega$ to obtain a polynomial $P(u,\Omega)$
    \State Verify $\forall u>0,\;\Omega\ge 1:\ P(u,\Omega)\ge 0$ via \texttt{Resolve}
    \If{\texttt{Resolve} returns \texttt{False}}
        \State \Return ``NOT certified for $n$''
    \EndIf
\EndFor
\For{every edge with free coordinate $r_k$, $k \le \lceil n/2 \rceil$}
    \State Construct the block quantities $L,R,\mu_\pm,\sigma_\pm^2$ and set $D,\delta$ \Comment{Section~\ref{sec:onedim-general}}
    \State Form the existence guard $|\delta| < D^2$ and the admissibility guards \eqref{eq:adm1}--\eqref{eq:adm2}
    \State Form the maximal value $q_{\max}$ via \eqref{eq:qmax}
    \State Clear denominators and negative powers of $\Omega$ to obtain a polynomial $H(u,\Omega)$
    \State Verify $\forall u>0,\;\Omega\ge 1:\ \bigl(|\delta|<D^2\bigr) \land \eqref{eq:adm1} \land \eqref{eq:adm2} \implies H(u,\Omega)\ge 0$
    \If{\texttt{Resolve} returns \texttt{False}}
        \State \Return ``NOT certified for $n$''
    \EndIf
\EndFor
\State \Return ``certified TRUE for $n$''
\end{algorithmic}
\end{algorithm}

\section{The case \texorpdfstring{$n=7$}{n=7} and the true bound}
\label{sec:n7}

The structural analysis developed in
Sections~\ref{sec:structure}~and~\ref{sec:onedim-general} is independent of
the degree \(n\). Hence, for any \(n\) the global maximizer of \(q(u;\cdot)\)
on \(\mathcal{B}\) must lie on a vertex or an edge, and the finite
enumeration strategy can be applied to test the conjecture for a given degree
or to compute the exact worst-case constant when it fails.

For \(n=7\) the conjecture is false~\cite{Shi2026}. In the light of
Theorem~\ref{thm:sharp-constant}, a particularly transparent counterexample can be
obtained with equally spaced collinear control points, for which the
reduction to the scalar function \(q\) is exact. Let
\[
n=7,
\qquad
\Omega=\alpha=\frac{11}{10},
\qquad
\omega_i=\alpha^{|i-4|},\quad i=0,\ldots,7,
\]
so that
\[
(\omega_0,\ldots,\omega_7)
=
(\alpha^4,\alpha^3,\alpha^2,\alpha,1,\alpha,\alpha^2,\alpha^3).
\]
The corresponding adjacent weight ratios are
\[
(r_1,\ldots,r_7)
=
\left(
\frac{10}{11},\frac{10}{11},\frac{10}{11},\frac{10}{11},
\frac{11}{10},\frac{11}{10},\frac{11}{10}
\right),
\]
so \(\mathbf{r}\) is a \emph{vertex} of \(\mathcal{B}=[\Omega^{-1},\Omega]^7\).
Choose the collinear control points
\[
\mathbf{P}_i=(i,0),
\qquad i=0,\ldots,7,
\]
so that \(\Delta_i=(1,0)\) and \(\max_{0\le i\le 6}\|\Delta_i\|=1\). Take
\[
u=\frac54,
\qquad
t=\frac{u}{1+u}=\frac59\in(0,1),
\]
an interior parameter, so that the reduction \eqref{eq:main-bound} applies
with equality by Theorem~\ref{thm:sharp-constant}; note that the endpoint cases can
never produce a violation, by Lemma~\ref{lem:endpoints}. With
\(a_i=\binom 7i\,\omega_i u^i\) and
\(\mathbb{P}(I=i)=a_i/\sum_{j=0}^7a_j\), all data are rational, so \(q\) is a
rational number that can be evaluated exactly. Exact rational arithmetic
gives
\[
\operatorname{Var}(I) \approx 1.902713,
\]
and consequently
\[
\begin{aligned}
q\left(
\frac54;\,
\frac{10}{11},\frac{10}{11},\frac{10}{11},\frac{10}{11},
\frac{11}{10},\frac{11}{10},\frac{11}{10}
\right)
&= \frac{(1+u)^2}{7u}\operatorname{Var}(I)
= \frac{81}{140}\operatorname{Var}(I) \\
&\approx 1.100855
> 1.1 = \Omega .
\end{aligned}
\]
The margin \(q-\Omega\approx 8.6\times10^{-4}\) is far beyond any rounding
uncertainty. Since the control points are equally spaced and
collinear, \(\|\mathbf{R}'(t)\| = 7\,q(u;\mathbf{r})\), whence
\[
\left\|\mathbf{R}'\!\left(\frac59\right)\right\|
\approx 7.70599
\;>\;
7\Omega\max_{0\le i\le6}\|\Delta_i\| = 7\cdot\frac{11}{10} = 7.7 .
\]
Thus the conjectured derivative bound fails for \(n=7\), even for equally
spaced collinear control points, and the failure necessarily occurs at an
interior parameter.

\section{Refining derivative bounds by source-curve subdivision}
\label{sec:source-curve-subdivision}
A global counterexample does not preclude the conjectured bound from
holding on each subcurve. Under source-curve subdivision, the local
weight-ratio bounds \(\Omega_I\) converge uniformly to \(1\) as the
mesh size tends to zero. Thus, provided that the conjecture holds
for the degree under consideration on some interval
\(1\le\Omega\le\Omega_0\) with \(\Omega_0>1\), sufficiently fine
subdivision ensures that every subcurve satisfies the conjectured
bound with its own factor \(\Omega_I\). The following theorem makes
this statement precise.

\begin{theorem}[Local derivative bounds under source-curve subdivision]
\label{thm:local-subdivision}
Let \(\mathbf{R}\) be a rational B\'ezier curve of degree \(n\) with
positive weights. For each subinterval \(I=[a,b]\) of a partition of
\([0,1]\), let
\[
  \mathbf{R}_I(s)=\mathbf{R}\bigl(a+(b-a)s\bigr),
  \qquad 0\le s\le1,
\]
be the corresponding subcurve obtained by homogeneous de Casteljau
subdivision. Denote its weights and control edges by \(w_{I,i}\) and
\(\Delta_{I,i}\), respectively, and set
\[
  \Omega_I
  :=\max_{0\le i<n}
    \max\left\{
      \frac{w_{I,i+1}}{w_{I,i}},
      \frac{w_{I,i}}{w_{I,i+1}}
    \right\}.
\]
If \(h\) is the mesh size of the partition, then
\[
  \max_I(\Omega_I-1)=O(h),
  \qquad
  \max_I K_n(\Omega_I)\longrightarrow 1
  \quad\text{as }h\to0.
\]
In particular, for every \(\varepsilon>0\), sufficiently fine
subdivision ensures that every subcurve satisfies
\[
  \|\mathbf{R}_I'(s)\|
  \le n(1+\varepsilon)\max_{0\le i<n}\|\Delta_{I,i}\|,
  \qquad 0\le s\le1.
\]

If, in addition, there exists \(\Omega_0>1\) such that
\(K_n(\Omega)\le\Omega\) for all \(1\le\Omega\le\Omega_0\), then
sufficiently fine subdivision yields the sharper bound
\begin{equation}\label{eq:local-subdivision-bound}
  \|\mathbf{R}_I'(s)\|
  \le n\,\Omega_I\max_{0\le i<n}\|\Delta_{I,i}\|,
  \qquad 0\le s\le1,
\end{equation}
on every subcurve.
\end{theorem}

\begin{proof}
Let \(W\) be the denominator polynomial of \(\mathbf{R}\), and let
\(W^\flat\) denote its blossom. The weights of the subcurve on
\(I=[a,b]\) are
\[
  w_{I,i}
  =W^\flat(\underbrace{a,\ldots,a}_{n-i},
           \underbrace{b,\ldots,b}_{i}),
  \qquad 0\le i\le n.
\]
Consecutive weights differ in only one blossom argument, which
changes from \(a\) to \(b\). Since \(W^\flat\) is multi-affine, its
first partial derivatives are bounded on \([0,1]^n\). Hence there
is a constant \(C\), depending only on the original weights, such
that
\[
  |w_{I,i+1}-w_{I,i}|\le C(b-a).
\]
Moreover, de Casteljau subdivision forms convex combinations of the
original weights, so \(w_{I,i}\ge w_{\min}:=\min_j w_j>0\).
Consequently,
\[
  0\le\Omega_I-1
  =\max_{0\le i<n}
    \frac{|w_{I,i+1}-w_{I,i}|}
         {\min\{w_{I,i},w_{I,i+1}\}}
  \le \frac{C}{w_{\min}}h.
\]
Thus \(\Omega_I\to1\) uniformly over the subcurves.

To pass from the ratios to the derivative factors, recall
\eqref{eq:K-as-sup-q}. In the parameter \(t=u/(1+u)\), the factor
\(q\) extends continuously to \(0\le t\le1\) by the endpoint
formulas. When \(\mathbf{r}=(1,\ldots,1)\), the index distribution
is binomial and the extended factor is identically \(1\).
Uniform continuity on
\([0,1]\times[1/2,2]^n\) therefore gives
\(K_n(\Omega)\to1\) as \(\Omega\downarrow1\). Since the convergence
\(\Omega_I\to1\) is uniform in \(I\), so is
\(K_n(\Omega_I)\to1\). Applying \eqref{eq:sharp-intro} to each
locally parameterized subcurve proves the first derivative bound.

Finally, if the stated bound holds on \([1,\Omega_0]\), choose the
mesh sufficiently small that \(\Omega_I\le\Omega_0\) for every
\(I\). Substituting \(K_n(\Omega_I)\le\Omega_I\) into
\eqref{eq:sharp-intro} gives \eqref{eq:local-subdivision-bound}.
\end{proof}

\section{Numerical analysis and examples}
\label{sec:numerical}
In this section, we present the numerical variant of Algorithm~\ref{alg:verify-degree}, implemented in MATLAB. Algorithm~\ref{alg:num-degree} differs from Algorithm~\ref{alg:verify-degree} in that the weights are prescribed, and we do not enumerate {the $2^n$ vertices and the $n\,2^{n-1}$ edges of the box $\mathcal{B}$}. Instead of exact quantifier elimination via \texttt{Resolve} in Lines 4 and 14, it employs grid sampling over a truncated domain, local optimization, and tolerance checking. To be precise, define
\[
\mathcal{D}=\{(u,\Omega):\;u_{\min}\le u\le u_{\max},\;1\le\Omega\le\Omega_{\max}\}
\]
as the truncated domain, and $\mathcal{G}=\mathcal{U}\times\mathcal{W}\subset\mathcal{D}$ as the grid, where $\mathcal{U}$ is a log-equidistant grid for $u$ ($N_u$ points), $\mathcal{W}$ is a grid for $\Omega$ ($N_\Omega$ points, refined near $\Omega=1$), and $\varepsilon>0$ is the tolerance. 

Using the function \(q\) in~\eqref{eq:q-def}, we define
\[
  B_q=nM\sup_{u>0}q(u,\mathbf r),
  \qquad
  \mathbf r=
  \left(\frac{\omega_1}{\omega_0},\ldots,
        \frac{\omega_n}{\omega_{n-1}}\right).
\]
The supremum is evaluated numerically with \(\mathbf r\) held
fixed. Thus, \(B_q\) is specific to the prescribed weights:
the identity \(K_n(\Omega)=\Omega\), when applicable, does not
imply that \(B_q=n\Omega M\) for each individual curve.

Moreover, we adopt several notations in the examples, whose meanings are specified as follows: $S$ denotes the exact supremum of the norm of the first derivative, {$n\Omega M$} represents the classical conjectured linear bound, {and $B_{\mathrm{sub}}(r)$ denotes the subdivision bound obtained after $r$ rounds of de Casteljau bisection of the original curve.} $B_{800}$ denotes the degree-elevation bound obtained by elevating the $2n$-degree derivative polynomial by $e=800$ orders. Excess (\%) quantifies the relative overestimation error of a given bound with respect to the true value $S$. It is defined as $\left(\dfrac{\text{Bound}}{S}-1\right)\times 100\%$. Time ($\mu\text{s}$) measures the computational cost (in microseconds) required to evaluate a specific bound.

 Examples 1 and 2 contrast our results with the counterexamples in \cite{Shi2026}. {It turns out that subdivision restores the validity of the previously violated estimates.} Example 3 uses randomly constructed rational Bézier curves to compare with the algorithm of \cite{Shi2026}.

\begin{algorithm}[H]
\caption{ComputeTightBound($\mathbf{P}, \omega$)}
\label{alg:num-degree}
\begin{algorithmic}[1]
\Require control points $\mathbf{P}_i$, weights $\omega_i > 0$, degree $n$
\Ensure Tight bound multiplier $q_{\max}$ and derivative bound

\State Compute adjacent weight ratios $r_i = \omega_i/\omega_{i-1}$, and the maximal ratio $\Omega$ {by Eq.~\eqref{eq:omega-def}}.
\State Compute the control polygon max length $M_\Delta = \max_i \|\Delta_i\|$.
\State Construct the normalized weights $p_i(u)$ and parameters {via Eq.~\eqref{eq:R-u}}.
\State Optimize the scalar function $q(u; \mathbf{r})$ {defined in Eq.~\eqref{eq:q-def}} over $u > 0$:
\State \quad 1. Grid search over $\ln u \in [\ln 10^{-6}, \ln 10^6]$.
\State \quad 2. Local refinement (e.g., Nelder-Mead) to find the maximum $q_{\max} = \max_{u>0} q(u; \mathbf{r})$.
\State Check the endpoint limits $\lim_{u\to0^+} q = r_1$ and $\lim_{u\to\infty} q = r_n^{-1}$ (see {Remark~\ref{rem:endpoint-limits}}), and update $q_{\max}$ if necessary.
\State \Return $q_{\max}$ and the bound $n q_{\max} M_\Delta$ {via Eq.~\eqref{eq:main-bound}}.
\end{algorithmic}
\end{algorithm}

\begin{example}[See also Example~2 in \cite{Shi2026}]
For each degree \(n=2,\ldots,30\), consider the rational
B\'ezier curve \(\mathbf R\) with weights and control points
\[
  \omega_i=\alpha^{|i-m|},
  \qquad
  \mathbf P_i=
  \left(i,\frac{1}{20}\sin\frac{\pi i}{n}\right),
  \qquad i=0,\ldots,n,
\]
where \(m=\lfloor n/2\rfloor\) and \(\alpha=11/10\).
Let
\[
  M=\max_{0\le i<n}\|\mathbf P_{i+1}-\mathbf P_i\|_2,
  \qquad
  S=\max_{0\le t\le1}\|\mathbf R'(t)\|_2.
\]
For this family, the adjacent-weight ratio bound is
\(\Omega=\alpha\).

Table~\ref{Table1:num} compares the numerically computed value
of \(S\) with four estimates: the conjectured value \(n\Omega M\);
the estimate \(B_{\mathrm{sub}}(1)\) obtained after one midpoint
subdivision of the source curve; the fixed-weight bound \(B_q\);
and the bound \(B_{800}\) obtained after \(800\) degree-elevation
steps. In computing \(B_{\mathrm{sub}}(1)\), each subcurve estimate
is multiplied by \(2\) to express it in the parameter of the
original curve, and the larger of the two values is retained. 
Table~\ref{Eax1:err_time} reports the signed relative excess
 of each estimate, together with its runtime in microseconds. A negative value indicates an
underestimate. Figures~\ref{fig11:tab3_fig} and~\ref{fig12:tab3_fig} show the
relative excess and runtime, respectively.

\begin{table}[htbp]
\centering
\caption{Absolute Bounds}\label{Table1:num}
\tiny
\begin{tabular}{@{}cccccc@{}}
\hline
$n$ & $S$ & $n\Omega M$ & $B_{\mathrm{sub}}(1)$ & $B_q$ & $B_{800}$ \\
\hline
2 & 2.095238 & 2.202748 & 2.195011 & 2.097856 & 2.095611 \\
3 & 3.186405 & 3.303092 & 3.372823 & 3.189194 & 3.187227 \\
4 & 4.292683 & 4.402749 & 4.497096 & 4.295365 & 4.294261 \\
5 & 5.415446 & 5.502375 & 5.723975 & 5.417712 & 5.417958 \\
6 & 6.556553 & 6.602062 & 6.868770 & 6.558602 & 6.560392 \\
7 & 7.706513 & 7.701812 & 8.135452 & 7.708291 & 7.711885 \\
8 & 8.875039 & 8.801611 & 9.297660 & 8.876663 & 8.882421 \\
9 & 10.049171 & 9.901447 & 10.599149 & 10.050620 & 10.058799 \\
10 & 11.241522 & 11.001313 & 11.776833 & 11.242864 & 11.253941 \\
11 & 12.437696 & 12.101200 & 13.109876 & 12.438917 & 12.453173 \\
12 & 13.651607 & 13.201105 & 14.301683 & 13.652750 & 13.670758 \\
13 & 14.868232 & 14.301024 & 15.663946 & 14.869288 & 14.891351 \\
14 & 16.102098 & 15.400953 & 16.868865 & 16.103095 & 16.129872 \\
15 & 17.337936 & 16.500892 & 18.258568 & 17.338867 & 17.370682 \\
16 & 18.590542 & 17.600837 & 19.475805 & 18.591426 & 18.629015 \\
17 & 19.844594 & 18.700789 & 20.891536 & 19.845427 & 19.889123 \\
18 & 21.114975 & 19.800746 & 22.120450 & 21.115770 & 21.166406 \\
19 & 22.386418 & 20.900708 & 23.561052 & 22.387172 & 22.445052 \\
20 & 23.673784 & 22.000673 & 24.801107 & 23.674508 & 23.740609 \\
21 & 24.961922 & 23.100641 & 26.265613 & 24.962613 & 25.037275 \\
22 & 26.265613 & 24.200613 & 27.516356 & 26.266278 & 26.350439 \\
23 & 27.569851 & 25.300586 & 29.003942 & 27.570488 & 27.664543 \\
24 & 28.889302 & 26.400562 & 30.264983 & 28.889917 & 28.994908 \\
25 & 30.209123 & 27.500540 & 31.774934 & 30.209714 & 30.326079 \\
26 & 31.543844 & 28.600519 & 33.045931 & 31.544417 & 31.673174 \\
27 & 32.878793 & 29.700500 & 34.577624 & 32.879346 & 33.021004 \\
28 & 34.228356 & 30.800483 & 35.858277 & 34.228892 & 34.384518 \\
29 & 35.578031 & 31.900466 & 37.411160 & 35.578550 & 35.748675 \\
30 & 36.942054 & 33.000451 & 38.701200 & 36.942559 & 37.128318 \\
\hline
\end{tabular}
\end{table}

\begin{table}[htbp]
\centering
\caption{Excess Errors and Timings}\label{Eax1:err_time}
\setlength{\tabcolsep}{4pt}
\tiny
\begin{tabular}{@{}c|cccc|cccc@{}}
\hline
& \multicolumn{4}{c|}{Relative excess (\%)} & \multicolumn{4}{c}{Runtime ($\mu$s)} \\
\cline{2-9}
$n$ & $n\Omega M$ & $B_{\mathrm{sub}}(1)$ & $B_q$ & $B_{800}$ & $n\Omega M$ & $B_{\mathrm{sub}}(1)$ & $B_q$ & $B_{800}$ \\
\hline
2 & +5.131 & +4.762 & +0.125 & +0.018 & 1.31 & 17.97 & 2223.26 & 4029.16 \\
3 & +3.662 & +5.850 & +0.088 & +0.026 & 1.56 & 38.63 & 3797.56 & 4933.66 \\
4 & +2.564 & +4.762 & +0.062 & +0.037 & 2.24 & 33.49 & 4227.06 & 6172.26 \\
5 & +1.605 & +5.697 & +0.042 & +0.046 & 2.28 & 39.09 & 3852.86 & 4884.56 \\
6 & +0.694 & +4.762 & +0.031 & +0.059 & 1.70 & 34.68 & 3401.86 & 5198.06 \\
7 & -0.061 & +5.566 & +0.023 & +0.070 & 1.96 & 37.69 & 3388.66 & 5606.26 \\
8 & -0.827 & +4.762 & +0.018 & +0.083 & 1.88 & 40.04 & 3087.86 & 5498.96 \\
9 & -1.470 & +5.473 & +0.014 & +0.096 & 1.76 & 41.08 & 2835.36 & 5034.56 \\
10 & -2.137 & +4.762 & +0.012 & +0.110 & 1.75 & 34.62 & 3321.86 & 5462.86 \\
11 & -2.705 & +5.404 & +0.010 & +0.124 & 1.67 & 40.80 & 3064.76 & 5369.96 \\
12 & -3.300 & +4.762 & +0.008 & +0.140 & 3.37 & 48.35 & 2868.46 & 5494.06 \\
13 & -3.815 & +5.352 & +0.007 & +0.155 & 1.61 & 45.83 & 2861.36 & 5535.46 \\
14 & -4.354 & +4.762 & +0.006 & +0.172 & 3.77 & 99.76 & 3440.96 & 6110.76 \\
15 & -4.828 & +5.310 & +0.005 & +0.189 & 1.80 & 48.72 & 3050.86 & 6036.46 \\
16 & -5.324 & +4.762 & +0.005 & +0.207 & 1.86 & 53.92 & 2506.26 & 5416.96 \\
17 & -5.764 & +5.276 & +0.004 & +0.224 & 1.51 & 50.87 & 2839.46 & 5524.36 \\
18 & -6.224 & +4.762 & +0.004 & +0.244 & 1.82 & 50.80 & 2507.26 & 5447.06 \\
19 & -6.637 & +5.247 & +0.003 & +0.262 & 1.83 & 58.14 & 2683.36 & 5609.96 \\
20 & -7.067 & +4.762 & +0.003 & +0.282 & 1.70 & 58.94 & 2547.26 & 5561.36 \\
21 & -7.456 & +5.223 & +0.003 & +0.302 & 1.78 & 63.31 & 2578.76 & 5687.06 \\
22 & -7.862 & +4.762 & +0.003 & +0.323 & 1.89 & 65.64 & 2851.06 & 5943.26 \\
23 & -8.231 & +5.202 & +0.002 & +0.343 & 1.94 & 55.08 & 2775.46 & 6756.96 \\
24 & -8.615 & +4.762 & +0.002 & +0.366 & 1.92 & 74.31 & 2986.76 & 8051.56 \\
25 & -8.966 & +5.183 & +0.002 & +0.387 & 2.10 & 79.96 & 3059.46 & 7135.66 \\
26 & -9.331 & +4.762 & +0.002 & +0.410 & 2.25 & 78.08 & 2696.06 & 6650.06 \\
27 & -9.667 & +5.167 & +0.002 & +0.433 & 2.13 & 80.09 & 2747.56 & 6746.56 \\
28 & -10.015 & +4.762 & +0.002 & +0.456 & 2.02 & 84.53 & 2710.36 & 6257.56 \\
29 & -10.337 & +5.152 & +0.001 & +0.480 & 1.66 & 69.72 & 2448.96 & 6331.66 \\
30 & -10.670 & +4.762 & +0.001 & +0.504 & 1.90 & 82.55 & 3192.36 & 6454.26 \\
\hline
\end{tabular}
\end{table}

\begin{figure}[htbp]
\centering
\includegraphics[width=0.6\textwidth]{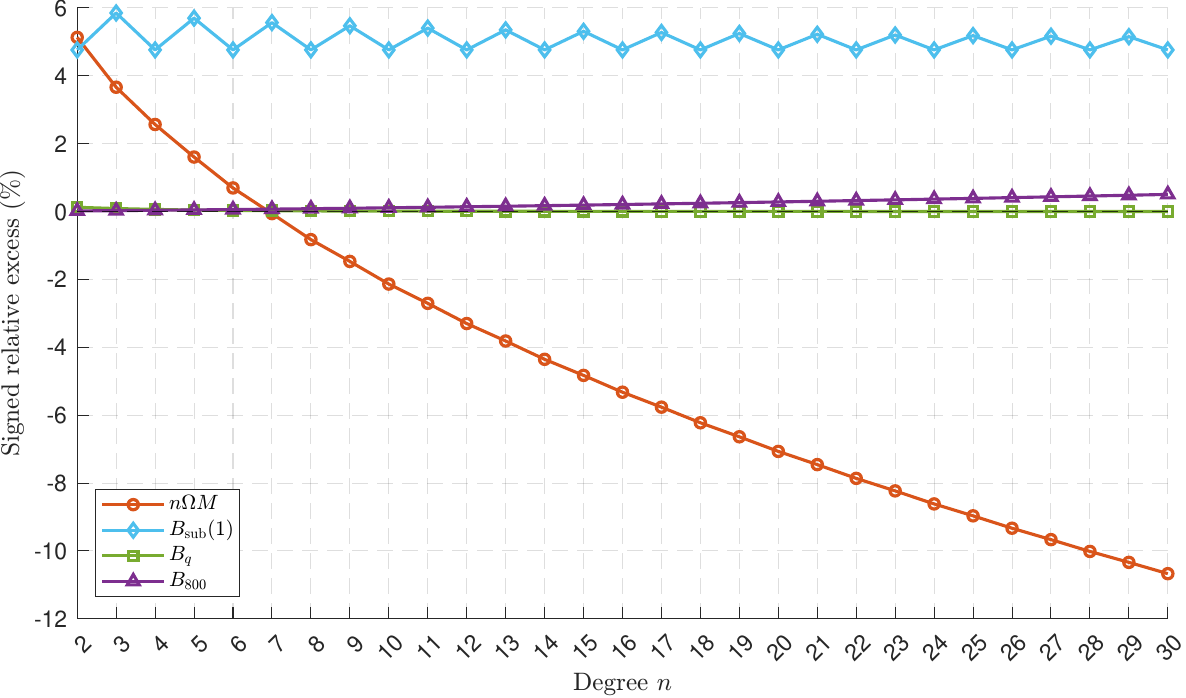} 
\caption{Excess Errors}
\label{fig11:tab3_fig}
\end{figure}

\begin{figure}[htbp]
\centering
\includegraphics[width=0.6\textwidth]{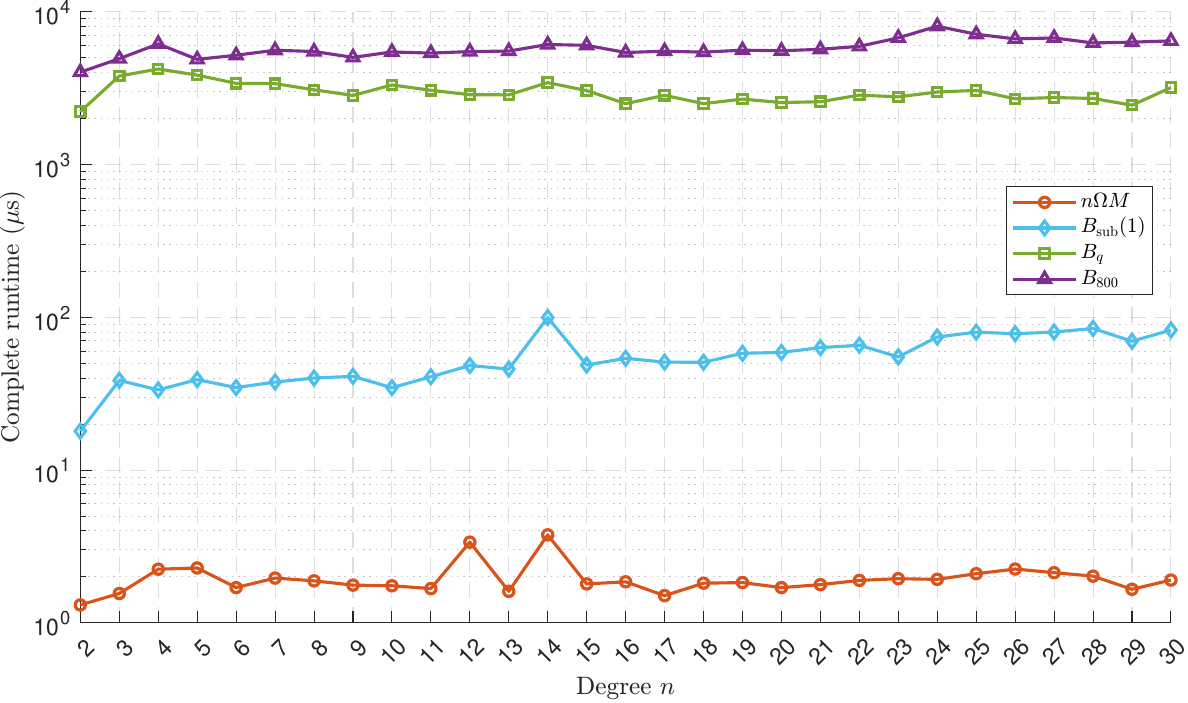} 
\caption{Timings}
\label{fig12:tab3_fig}
\end{figure}

\end{example}

\begin{example} (Also Example 3 in \cite{Shi2026})   For rational Bézier curves spanning degrees $n = 2$ to $30$, we employed the following test configuration:

\begin{itemize}
\item \textbf{Weights:} $\omega_i = 2^{-i}$ for $i = 0, \ldots, n-1$, with $\omega_n = 2^{-(n-2)}$.
\item \textbf{Control points:} $\mathbf{r}_i = (i, 0) \in \mathbb{R}^2$ for $i = 0, \ldots, n$.
\item  \textbf{degree elevation number:} $e=800$.
\end{itemize}

Table~\ref{Table2:num} compares the numerically computed value
of \(S\) with four estimates: the conjectured value \(n\Omega M\);
the estimate \(B_{\mathrm{sub}}(1)\) obtained after one midpoint
subdivision of the source curve; the fixed-weight bound \(B_q\);
and the bound \(B_{800}\) obtained after \(800\) degree-elevation
steps. In computing \(B_{\mathrm{sub}}(1)\), each subcurve estimate
is multiplied by \(2\) to express it in the parameter of the
original curve, and the larger of the two values is retained. 
Table~\ref{Eax2:err_time} reports the signed relative excess
 of each estimate, together with its runtime in microseconds. A negative value indicates an
underestimate. For this example, the equality \(S=B_q\) is exact.
Indeed, since \(\mathbf r_i=(i,0)\), we have \(M=1\), and the
first coordinate of the curve is the mean of the index \(i\)
under the normalized rational Bernstein weights. With
\(u=t/(1-t)\), differentiating this mean gives
\(\|\mathbf R'(t)\|=nq(u,\mathbf r)\) for \(0<t<1\).
Taking the supremum therefore yields \(S=B_q\).
In contrast, the conjectured estimate \(n\Omega M=2n\)
first falls below \(S\) at \(n=11\) among the degrees tested.
Figures~\ref{fig21:tab3_fig} and~\ref{fig22:tab3_fig} show the
relative excess and runtime, respectively.
\begin{table}[htbp]
\centering
\caption{Absolute Bounds}\label{Table2:num}
\tiny
\begin{tabular}{@{}cccccc@{}}
\hline
$n$ & $S$ & $n\Omega M$ & $B_{\mathrm{sub}}(1)$ & $B_q$ & $B_{800}$ \\
\hline
2 & 2.666667 & 4.000000 & 3.555556 & 2.666667 & 2.669989 \\
3 & 4.466444 & 6.000000 & 6.400000 & 4.466444 & 4.475724 \\
4 & 6.464102 & 8.000000 & 8.960000 & 6.464102 & 6.482606 \\
5 & 8.571204 & 10.000000 & 11.714286 & 8.571204 & 8.602066 \\
6 & 10.748531 & 12.000000 & 14.282927 & 10.748531 & 10.794831 \\
7 & 12.974278 & 14.000000 & 16.754098 & 12.974278 & 13.038869 \\
8 & 15.235043 & 16.000000 & 19.182466 & 15.235043 & 15.320609 \\
9 & 17.522048 & 18.000000 & 21.593419 & 17.522048 & 17.631514 \\
10 & 19.829270 & 20.000000 & 23.997562 & 19.829270 & 19.965057 \\
11 & 22.152423 & 22.000000 & 26.399106 & 22.152423 & 22.317326 \\
12 & 24.488370 & 24.000000 & 28.799675 & 24.488370 & 24.684655 \\
13 & 26.834755 & 26.000000 & 31.199883 & 26.834755 & 27.064674 \\
14 & 29.189773 & 28.000000 & 33.599958 & 29.189773 & 29.456223 \\
15 & 31.552017 & 30.000000 & 35.999985 & 31.552017 & 31.857283 \\
16 & 33.920374 & 32.000000 & 38.399995 & 33.920374 & 34.265998 \\
17 & 36.293948 & 34.000000 & 40.799998 & 36.293948 & 36.684419 \\
18 & 38.672014 & 36.000000 & 43.199999 & 38.672014 & 39.108389 \\
19 & 41.053973 & 38.000000 & 45.600000 & 41.053973 & 41.537651 \\
20 & 43.439332 & 40.000000 & 48.000000 & 43.439332 & 43.973884 \\
21 & 45.827677 & 42.000000 & 50.400000 & 45.827677 & 46.412722 \\
22 & 48.218659 & 44.000000 & 52.800000 & 48.218659 & 48.858036 \\
23 & 50.611982 & 46.000000 & 55.200000 & 50.611982 & 51.308384 \\
24 & 53.007393 & 48.000000 & 57.600000 & 53.007393 & 53.765253 \\
25 & 55.404673 & 50.000000 & 60.000000 & 55.404673 & 56.223918 \\
26 & 57.803635 & 52.000000 & 62.400000 & 57.803635 & 58.684612 \\
27 & 60.204115 & 54.000000 & 64.800000 & 60.204115 & 61.151221 \\
28 & 62.605970 & 56.000000 & 67.200000 & 62.605970 & 63.622406 \\
29 & 65.009074 & 58.000000 & 69.600000 & 65.009074 & 66.093333 \\
30 & 67.413317 & 60.000000 & 72.000000 & 67.413317 & 68.566179 \\
\hline
\end{tabular}
\end{table}

\begin{table}[htbp]
\centering
\caption{Excess Errors and Timings}\label{Eax2:err_time}
\setlength{\tabcolsep}{4pt}
\tiny
\begin{tabular}{@{}c|cccc|cccc@{}}
\hline
& \multicolumn{4}{c|}{Relative excess (\%)} & \multicolumn{4}{c}{Runtime ($\mu$s)} \\
\cline{2-9}
$n$ & $n\Omega M$ & $B_{\mathrm{sub}}(1)$ & $B_q$ & $B_{800}$ & $n\Omega M$ & $B_{\mathrm{sub}}(1)$ & $B_q$ & $B_{800}$ \\
\hline
2 & +50.000 & +33.333 & +0.000 & +0.125 & 1.77 & 18.34 & 3305.22 & 5208.82 \\
3 & +34.335 & +43.291 & +0.000 & +0.208 & 1.91 & 26.35 & 3099.42 & 4181.42 \\
4 & +23.760 & +38.612 & +0.000 & +0.286 & 1.51 & 24.99 & 2978.62 & 4718.72 \\
5 & +16.670 & +36.670 & +0.000 & +0.360 & 2.26 & 26.03 & 3382.42 & 5376.32 \\
6 & +11.643 & +32.883 & +0.000 & +0.431 & 1.42 & 25.92 & 3472.82 & 4977.72 \\
7 & +7.906 & +29.133 & +0.000 & +0.498 & 2.10 & 32.29 & 3097.82 & 4643.92 \\
8 & +5.021 & +25.910 & +0.000 & +0.562 & 1.63 & 34.16 & 2963.82 & 4660.02 \\
9 & +2.728 & +23.236 & +0.000 & +0.625 & 1.28 & 29.60 & 3040.42 & 4846.92 \\
10 & +0.861 & +21.021 & +0.000 & +0.685 & 1.66 & 38.60 & 2971.82 & 4871.12 \\
11 & -0.688 & +19.170 & +0.000 & +0.744 & 1.56 & 40.74 & 2911.22 & 4520.32 \\
12 & -1.994 & +17.606 & +0.000 & +0.802 & 1.42 & 39.86 & 2460.32 & 5118.22 \\
13 & -3.111 & +16.267 & +0.000 & +0.857 & 1.63 & 46.28 & 2776.12 & 5120.62 \\
14 & -4.076 & +15.109 & +0.000 & +0.913 & 1.62 & 45.89 & 2809.32 & 4683.02 \\
15 & -4.919 & +14.097 & +0.000 & +0.967 & 1.37 & 45.42 & 2914.02 & 5384.62 \\
16 & -5.661 & +13.206 & +0.000 & +1.019 & 1.58 & 51.50 & 2526.62 & 5395.02 \\
17 & -6.320 & +12.415 & -0.000 & +1.076 & 1.34 & 44.98 & 2556.92 & 5223.82 \\
18 & -6.909 & +11.709 & +0.000 & +1.128 & 1.59 & 55.91 & 3241.42 & 5588.12 \\
19 & -7.439 & +11.073 & +0.000 & +1.178 & 1.45 & 54.39 & 2644.92 & 5964.72 \\
20 & -7.918 & +10.499 & +0.000 & +1.231 & 1.50 & 58.94 & 2739.92 & 5803.62 \\
21 & -8.352 & +9.977 & +0.000 & +1.277 & 1.38 & 57.06 & 2560.72 & 5744.52 \\
22 & -8.749 & +9.501 & -0.000 & +1.326 & 1.54 & 64.59 & 2975.22 & 6206.62 \\
23 & -9.112 & +9.065 & +0.000 & +1.376 & 1.67 & 68.50 & 2938.12 & 5504.22 \\
24 & -9.447 & +8.664 & +0.000 & +1.430 & 1.56 & 57.88 & 2972.02 & 6085.82 \\
25 & -9.755 & +8.294 & +0.000 & +1.479 & 1.40 & 59.68 & 2865.22 & 6037.32 \\
26 & -10.040 & +7.952 & +0.000 & +1.524 & 1.52 & 71.66 & 2794.62 & 6827.72 \\
27 & -10.305 & +7.634 & -0.000 & +1.573 & 1.48 & 71.36 & 2818.82 & 5960.82 \\
28 & -10.552 & +7.338 & -0.000 & +1.624 & 1.59 & 74.15 & 2888.62 & 6273.72 \\
29 & -10.782 & +7.062 & +0.000 & +1.668 & 1.73 & 81.42 & 3282.02 & 6873.02 \\
30 & -10.997 & +6.804 & +0.000 & +1.710 & 1.73 & 85.14 & 3331.52 & 7140.92 \\
\hline
\end{tabular}
\end{table}

\begin{figure}[htbp]
\centering
\includegraphics[width=0.6\textwidth]{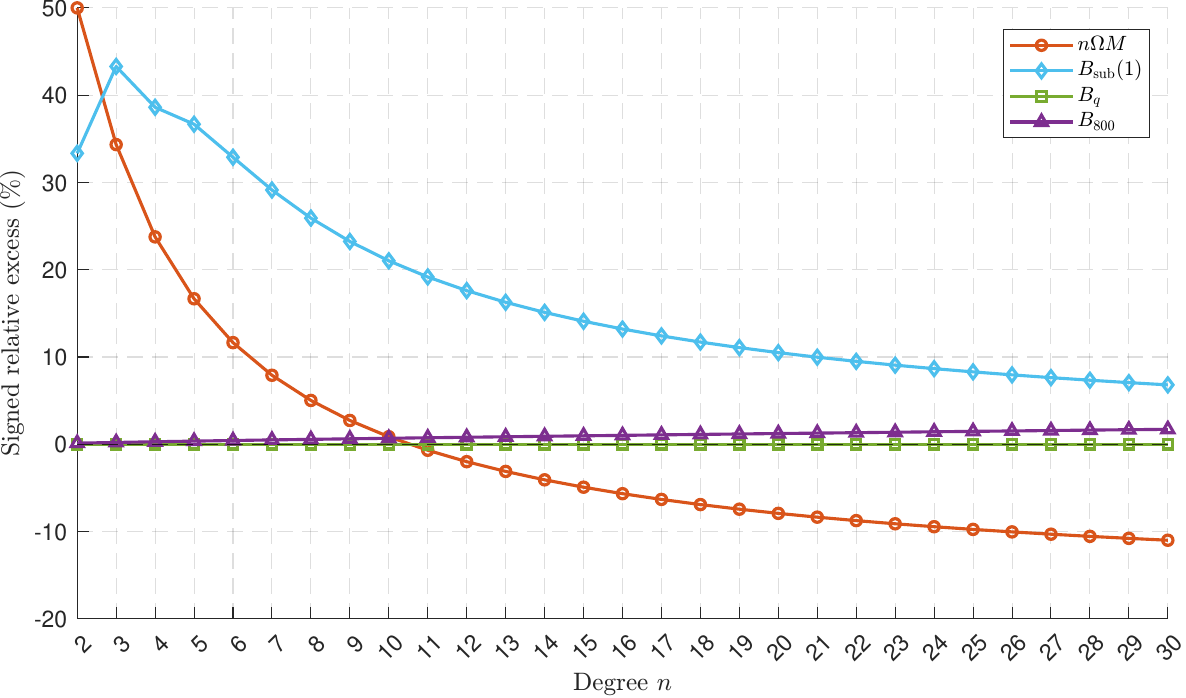} 
\caption{Excess Errors}
\label{fig21:tab3_fig}
\end{figure}

\begin{figure}[htbp]
\centering
\includegraphics[width=0.6\textwidth]{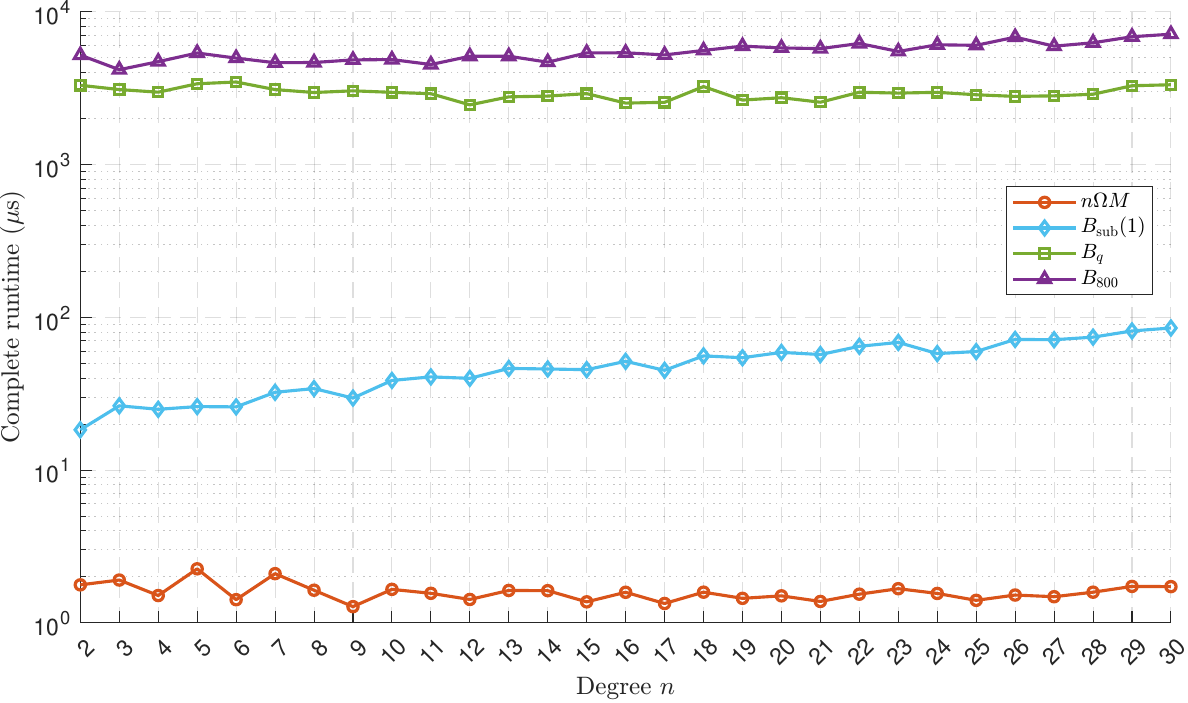} 
\caption{Timings}
\label{fig22:tab3_fig}
\end{figure}

\end{example}

\begin{example}
 
To rigorously evaluate the robustness of various derivative upper bound algorithms under extreme geometric configurations and drastic weight fluctuations, we generate a set of challenging rational Bézier test curves for degrees \(n \in [2, 30]\) using independent and identically distributed random variables. The control points \(\mathbf{P}_i = (x_i, y_i)\) are sampled independently from a uniform distribution over a large domain, namely \(x_i, y_i \sim \mathcal{U}(-100, 100)\). Such wide-range spatial sampling readily yields control polygons with sharp turns, extreme stretching and even self-intersections, which provides a highly non-trivial geometric setting for all bounding algorithms.

Variation in weights directly governs the ill-conditionedness of rational curves. We generate the weights \(w_i\) via a log-uniform distribution. We first sample \(u_i \sim \mathcal{U}(0, 1)\), and then apply the exponential map \(w_i = 10^{2u_i - 1}\). This construction ensures each weight \(w_i\) lies within \([0.1, 10]\), spanning two orders of magnitude. Under this setup, the theoretical peak value of the maximum adjacent weight ratio \(\Omega\) can reach 100.

This extreme initialization enables thorough examination of the looseness of theoretical bounds under non-uniform rational conditions. Highly fluctuating weights and distorted geometric data are introduced to reveal the over-penalization of linear bounds with respect to \(\Omega\), and further highlight the prominent advantages of subdivision and degree-elevation methods in local approximation accuracy and numerical stability.

Table~\ref{Table3:num} compares the numerically computed value
of \(S\) with four estimates: the conjectured value \(n\Omega M\);
the estimate {\(B_{\mathrm{sub}}(5)\) obtained after five rounds of
de Casteljau bisection} of the source curve; the fixed-weight bound \(B_q\);
and the bound \(B_{800}\) obtained after \(800\) degree-elevation
steps. In computing {\(B_{\mathrm{sub}}(5)\)}, {each of the \(2^5=32\) subcurve estimates
is multiplied by \(32\)} to express it in the parameter of the
original curve, and {the largest of the \(32\) values} is retained. 
Table~\ref{Eax3:err_time} reports the signed relative excess
 of each estimate, together with its runtime in microseconds. A negative value indicates an
underestimate. {Figure~\ref{fig31:tab3_fig} displays typical randomly generated rational
B\'ezier curves.} Figures~\ref{fig32:tab3_fig} and~\ref{fig33:tab3_fig} show the
relative excess and {the} runtime, respectively.

\begin{table}[htbp]
\centering
\caption{Absolute Bounds} \label{Table3:num}
\tiny
\begin{tabular}{@{}cccccc@{}}
\hline
$n$ & $S$ & $n\Omega M$ & $B_{\mathrm{sub}}(5)$ & $B_q$ & $B_{800}$ \\
\hline
2 & 229.004449 & 556.808158 & 231.102514 & 297.880031 & 229.288849 \\
3 & 1407.863396 & 1501.095036 & 1408.965133 & 1407.863396 & 1407.863396 \\
4 & 3606.169746 & 5097.848864 & 3606.169746 & 5097.848864 & 3606.169746 \\
5 & 557.831240 & 2494.783488 & 557.831240 & 1258.958800 & 557.831240 \\
6 & 1620.013461 & 41761.338214 & 1620.013461 & 1620.013461 & 1620.013461 \\
7 & 3055.263530 & 16255.844695 & 3055.263530 & 3468.130456 & 3055.263530 \\
8 & 1681.842499 & 32870.907772 & 1681.842499 & 6783.702374 & 1681.842499 \\
9 & 1103.808981 & 66555.266956 & 1115.640266 & 3563.805116 & 1132.177168 \\
10 & 8132.343686 & 62652.304176 & 8132.343686 & 12734.183416 & 8132.343686 \\
11 & 13185.383661 & 56391.383218 & 13185.383661 & 16109.017938 & 13185.383661 \\
12 & 1462.731491 & 168075.678693 & 1535.299115 & 6780.480933 & 1489.231161 \\
13 & 2127.205980 & 78515.351006 & 2311.786348 & 7730.475755 & 2198.415906 \\
14 & 3129.987920 & 106907.376186 & 3129.987920 & 6414.062333 & 3129.987920 \\
15 & 14339.789767 & 67823.799742 & 14339.789767 & 52901.322513 & 14339.789767 \\
16 & 722.662209 & 113905.882485 & 734.406964 & 7794.068292 & 750.613026 \\
17 & 20905.642644 & 118846.247134 & 20905.642644 & 20905.642644 & 20905.642644 \\
18 & 5733.503154 & 41262.785207 & 5733.503154 & 6726.845289 & 5733.503154 \\
19 & 7033.446817 & 144663.451557 & 7033.446817 & 14039.874044 & 7033.446817 \\
20 & 11157.618482 & 167743.297251 & 11157.618482 & 20651.987209 & 11157.618482 \\
21 & 1726.012819 & 230654.370186 & 1775.303428 & 10215.893855 & 1823.295451 \\
22 & 14608.282422 & 231797.374526 & 14608.282422 & 53544.383216 & 14608.282422 \\
23 & 20944.754475 & 84015.932864 & 20944.754475 & 20944.754475 & 20944.754475 \\
24 & 29686.293051 & 222065.594863 & 29686.293051 & 43268.405937 & 29686.293051 \\
25 & 4488.614566 & 278135.817879 & 4713.370673 & 10619.624567 & 4771.591950 \\
26 & 5540.570089 & 257019.116250 & 5540.570089 & 13865.256568 & 5540.570089 \\
27 & 38265.561107 & 363381.712700 & 38265.561107 & 42306.694003 & 38265.561107 \\
28 & 4755.082522 & 133620.524349 & 4828.791594 & 13374.879845 & 4755.082522 \\
29 & 11787.088334 & 415314.493220 & 11787.088334 & 16284.127208 & 11787.088334 \\
30 & 68950.084844 & 249431.049862 & 68950.084844 & 73829.111914 & 68950.084844 \\
\hline
\end{tabular}
\end{table}

\begin{table}[htbp]
\centering
\caption{Excess Errors and Timings}\label{Eax3:err_time}
\setlength{\tabcolsep}{4pt}
\tiny
\begin{tabular}{@{}c|cc|cc|cc|cc@{}}
\hline
& \multicolumn{4}{c|}{Relative excess (\%)} & \multicolumn{4}{c}{Runtime ($\mu$s)} \\
\cline{2-9}
$n$ & $n\Omega M$ & $B_{\mathrm{sub}}(5)$ & $B_q$ & $B_{800}$ & $n\Omega M$ & $B_{\mathrm{sub}}(5)$ & $B_q$ & $B_{800}$ \\
\hline
2 & +143.143 & +0.916 & +30.076 & +0.124 & 1.25 & 306.88 & 2451.58 & 3963.98 \\
3 & +6.622 & +0.078 & +0.000 & +0.000 & 1.56 & 413.14 & 3016.48 & 4072.98 \\
4 & +41.365 & +0.000 & +41.365 & +0.000 & 1.54 & 560.45 & 3138.08 & 5003.68 \\
5 & +347.229 & -0.000 & +125.688 & +0.000 & 1.70 & 686.55 & 3170.98 & 4687.38 \\
6 & +2477.839 & -0.000 & +0.000 & +0.000 & 1.54 & 643.50 & 2895.78 & 4479.28 \\
7 & +432.060 & -0.000 & +13.513 & +0.000 & 1.68 & 746.20 & 3121.58 & 5292.78 \\
8 & +1854.458 & +0.000 & +303.349 & +0.000 & 1.60 & 765.15 & 2841.28 & 4649.68 \\
9 & +5929.600 & +1.072 & +222.864 & +2.570 & 1.58 & 817.45 & 2855.28 & 4858.68 \\
10 & +670.409 & +0.000 & +56.587 & +0.000 & 1.60 & 918.40 & 3158.18 & 5536.28 \\
11 & +327.681 & +0.000 & +22.173 & +0.000 & 2.97 & 1014.78 & 3370.48 & 5354.08 \\
12 & +11390.535 & +4.961 & +363.549 & +1.812 & 1.84 & 1115.88 & 4124.28 & 5515.68 \\
13 & +3591.008 & +8.677 & +263.410 & +3.348 & 3.40 & 1183.68 & 2970.08 & 5360.88 \\
14 & +3315.584 & +0.000 & +104.923 & +0.000 & 1.43 & 1043.28 & 2652.88 & 5276.08 \\
15 & +372.976 & +0.000 & +268.913 & +0.000 & 1.84 & 1284.58 & 2785.78 & 5583.68 \\
16 & +15661.981 & +1.625 & +978.522 & +3.868 & 1.86 & 1323.08 & 2772.08 & 5480.78 \\
17 & +468.489 & -0.000 & +0.000 & +0.000 & 1.82 & 1387.58 & 2306.78 & 5091.18 \\
18 & +619.678 & +0.000 & +17.325 & +0.000 & 1.78 & 1398.08 & 2573.88 & 5456.48 \\
19 & +1956.793 & +0.000 & +99.616 & +0.000 & 1.90 & 1494.38 & 2946.68 & 5789.78 \\
20 & +1403.397 & +0.000 & +85.093 & +0.000 & 1.59 & 1375.58 & 2438.18 & 5595.38 \\
21 & +13263.422 & +2.856 & +491.878 & +5.636 & 1.93 & 1664.88 & 2861.68 & 5922.78 \\
22 & +1486.753 & -0.000 & +266.534 & +0.000 & 1.95 & 1917.28 & 3035.08 & 5653.68 \\
23 & +301.131 & +0.000 & +0.000 & +0.000 & 1.99 & 1805.38 & 2500.58 & 5948.88 \\
24 & +648.041 & -0.000 & +45.752 & +0.000 & 2.07 & 1893.88 & 2674.48 & 6410.78 \\
25 & +6096.474 & +5.007 & +136.590 & +6.304 & 1.67 & 1757.78 & 2714.18 & 6179.48 \\
26 & +4538.857 & -0.000 & +150.250 & +0.000 & 2.09 & 2085.28 & 2699.48 & 6530.18 \\
27 & +849.631 & -0.000 & +10.561 & +0.000 & 2.10 & 2089.88 & 2876.38 & 6406.88 \\
28 & +2710.057 & +1.550 & +181.275 & +0.000 & 2.04 & 2240.48 & 2825.58 & 6282.08 \\
29 & +3423.470 & +0.000 & +38.152 & +0.000 & 2.05 & 2232.68 & 3039.28 & 6558.58 \\
30 & +261.756 & +0.000 & +7.076 & +0.000 & 2.14 & 2306.48 & 2864.38 & 6114.28 \\
\hline
\end{tabular}
\end{table}

\begin{figure}[htbp]
\centering
\includegraphics[width=0.85\textwidth]{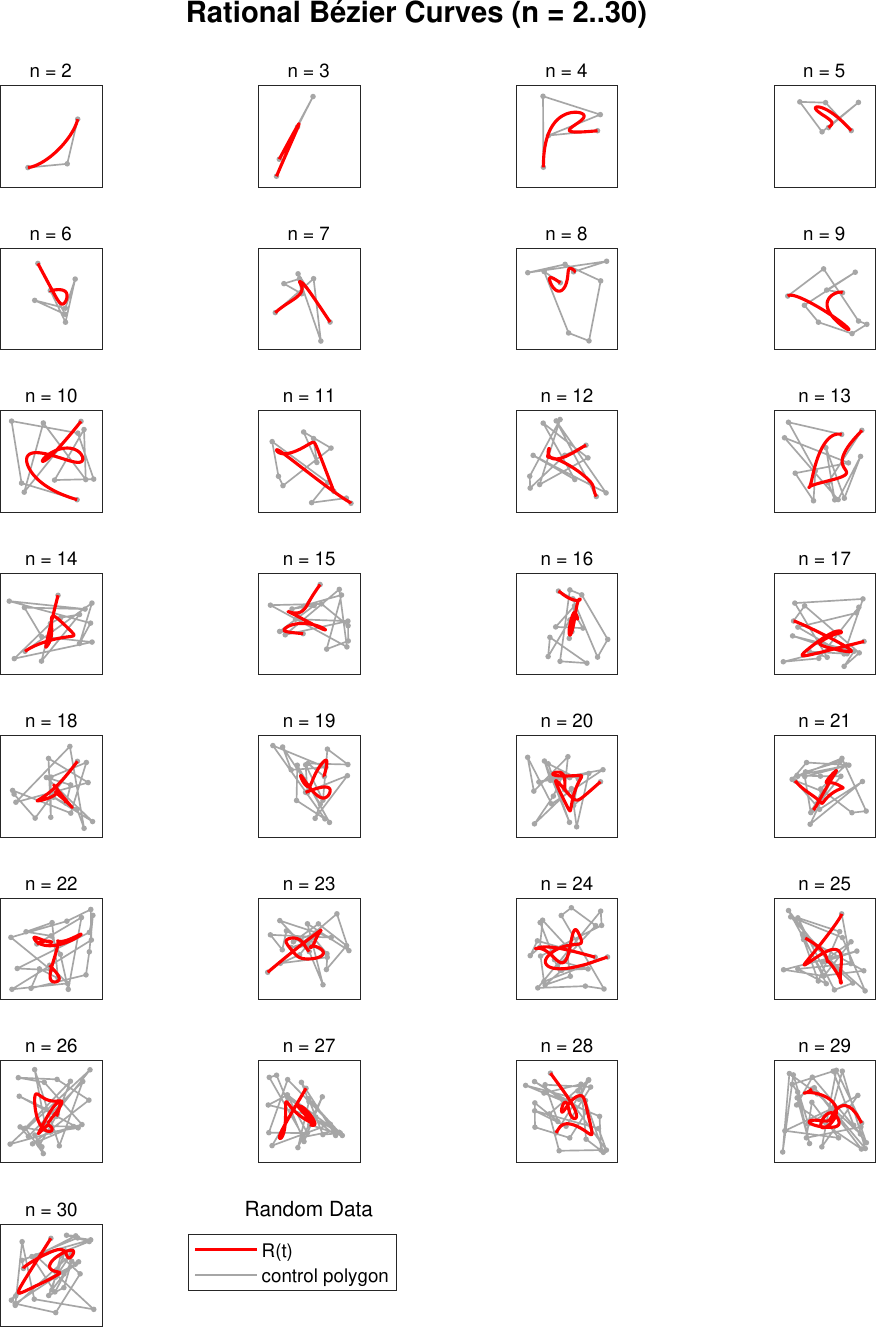} 
\caption{{Random rational B\'ezier curves of degrees \(2\)--\(30\) with their control polygons.}}
\label{fig31:tab3_fig}
\end{figure}

\begin{figure}[htbp]
\centering
\includegraphics[width=0.6\textwidth]{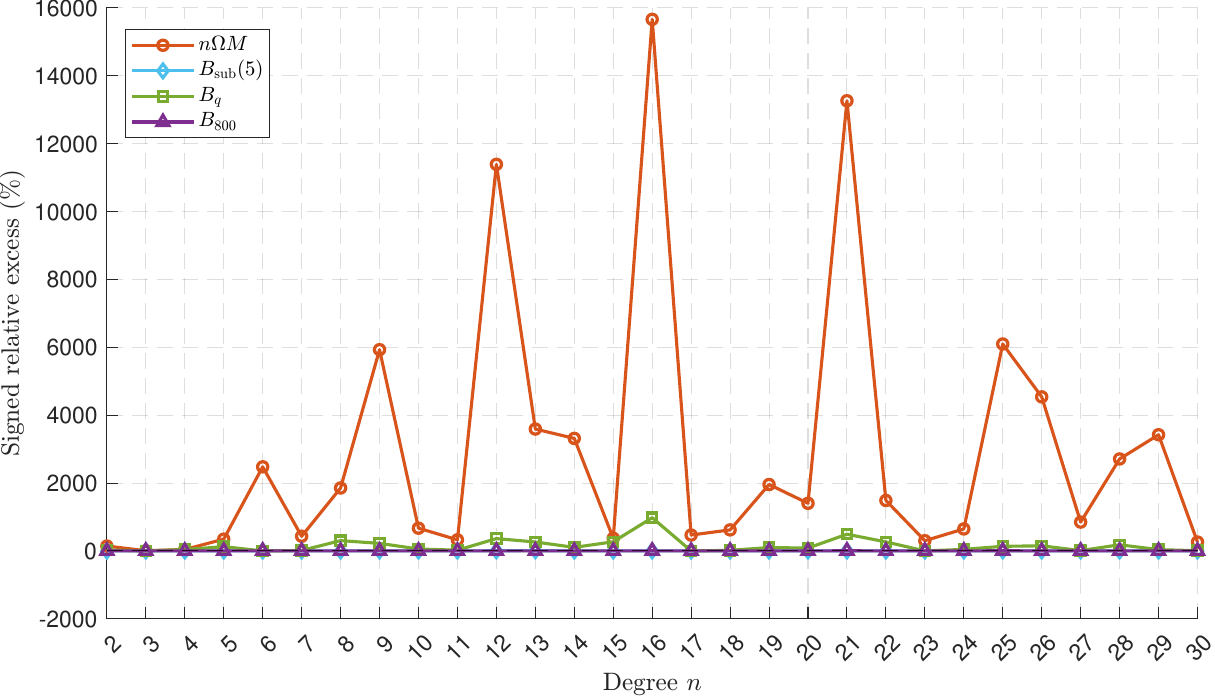} 
\caption{{Relative excess of the four bounds for the random family of Example 3.}}
\label{fig32:tab3_fig}
\end{figure}

\begin{figure}[htbp]
\centering
\includegraphics[width=0.6\textwidth]{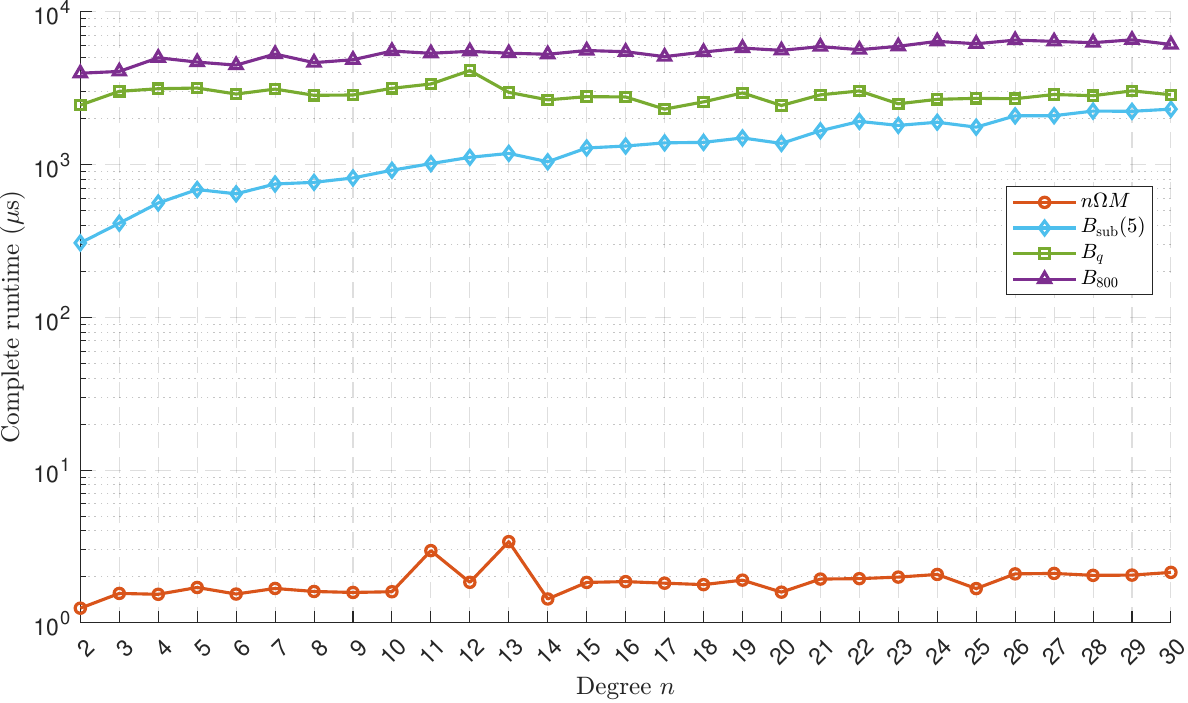} 
\caption{{Runtime per bound evaluation for the random family of Example 3.}}
\label{fig33:tab3_fig}
\end{figure}

\end{example}

\clearpage
\section{Conclusion}
\label{sec:conclusion}
This paper has established the first-derivative bound for rational Bézier curves of arbitrary degree.
The coefficient $K_n$ arises from a variance-based formulation of the rational basis functions.
This formulation provides a common setting for the general estimate, the remaining low-degree cases of the conjecture of Li et al., and the analysis of higher-degree behavior.

Within this framework, the conjectured bounds for degrees five and six have been proved as special cases.
Together with the previously established results, these cases complete the verification of the conjecture for $n\le6$.
The proof combines the classical endpoint derivative formulas with an interior optimization problem in adjacent weight ratios.
A block argument reduces the maximization to vertices and one-dimensional faces of the admissible box, leaving finitely many polynomial inequalities to verify.
The degree-seven counterexample identifies a limitation of the original conjectured estimate and motivates the more general coefficient $K_n$.

We have also proved that sufficiently fine subdivision yields subcurves on which the conjectured estimate is valid.
Each local estimate uses the weights and control edges of the corresponding subcurve and is understood with respect to its normalized parameter.
After the appropriate parameter rescaling, these estimates give a subdivision-based bound for the derivative of the original curve.
The subdivision result therefore extends the practical use of the conjectured form to higher-degree curves through their refined representations.

The numerical examples illustrate the behavior of the variance-based and subdivision-based bounds for the test families considered and provide comparisons with an existing method.
Further work may investigate more efficient evaluation of $K_n$ and adaptive subdivision strategies that balance the quality of the derivative bound against computational cost.

\bibliographystyle{elsarticle-num}
\bibliography{mybibfile-g}

\end{document}